\documentclass[preprint,sort&compress]{elsarticle}
\usepackage[utf8]{inputenc}

\usepackage[symbol]{footmisc}
\usepackage{orcidlink}
\usepackage{amsmath,amssymb,amsthm}
\usepackage{bm}
\usepackage{multirow}
\usepackage[algoruled,boxed,lined]{algorithm2e}
\usepackage{algpseudocode}

\usepackage{hyperref}
\usepackage[misc]{ifsym} 

\usepackage{color}
\usepackage{graphicx}
\graphicspath{{./figures/}}

\usepackage{amsfonts}
\usepackage{mathtools}

\usepackage[mathscr]{eucal}

\usepackage{caption}
\usepackage{bbm}

\newtheorem{theorem}{Theorem}
\newtheorem{lemma}{Lemma}

\newtheorem{remark}{Remark}
\newtheorem{assumption}{Assumption}
\newtheorem{example}{Example}

\providecommand{\E}[1]{{\ensuremath{\mathbb{E}}\mspace{-2mu}\left[#1\right]}}    

\DeclarePairedDelimiter\abs{\lvert}{\rvert}%
\DeclarePairedDelimiter\norm{\lVert}{\rVert}%

\DeclareMathOperator*{\argmin}{arg\,min}

\begin{document}

\title{Nonasymptotic Convergence Rate of Quasi-Monte Carlo: Applications to Linear Elliptic PDEs with Lognormal Coefficients and Importance Samplings}

\author[CEMSE]{Yang~Liu\corref{cor1}\fnref{fn1}\,\orcidlink{0000-0003-0778-3872}}
\ead{yang.liu.3@kaust.edu.sa}

\author[CEMSE]{Ra\'{u}l~Tempone\fnref{fn1,fn2}\,\orcidlink{0000-0003-1967-4446}}
\address[CEMSE]{Computer, Electrical and Mathematical Sciences and Engineering,\\
  4700 King Abdullah University of Science and Technology (KAUST),\\
  Thuwal 23955-6900, Kingdom of Saudi Arabia}

\cortext[cor1]{Corresponding author}
\fntext[fn1]{KAUST SRI Center for Uncertainty Quantification in Computational Science and Engineering}
\fntext[fn2]{Alexander von Humboldt Professor in Mathematics for Uncertainty
  Quantification, RWTH Aachen University, 52062 Aachen, Germany}

\markboth{Liu et al.}{Non-asymptotic convergence rate of the quasi-Monte Carlo method}

\begin{abstract}
  {This} study analyzes the {nonasymptotic} convergence behavior of {the} quasi-Monte Carlo (QMC) {method} with applications to linear elliptic {partial differential equations} (PDEs) with lognormal coefficients. {Building upon the error analysis presented in}~{(Owen, 2006)}, we derive a nonasymptotic convergence estimate {depending} on the specific integrands, {the input dimensionality}, and {the finite number of samples used in the QMC quadrature}. {We discuss the effects of the variance and dimensionality of the input random variable}. {Then, we} apply the {QMC method} with importance sampling (IS) {to approximate} deterministic, real-valued, bounded linear functionals that depend on the solution of a linear elliptic PDE with a lognormal diffusivity coefficient in bounded domains of $\mathbb{R}^d$, where the random coefficient is modeled as {a} stationary Gaussian random field parameterized by the trigonometric and wavelet-type basis. We propose two types of IS distributions, analyze their effects on the QMC convergence rate, and {observe the improvements}. 
\end{abstract}

\maketitle

\section{Introduction}
\label{sec:intro}
The quasi-Monte Carlo (QMC) method {computes} the expectation of a random variable using deterministic low-discrepancy sequences. For regular integrands, {the} QMC {method} offers a convergence rate of $\mathcal{O}(n^{-1 + \epsilon})$ for $\epsilon>0$, {surpassing} the Monte Carlo convergence rate of $\mathcal{O}(n^{-1/2})$. However, the effectiveness of {the} QMC {method} depends on the regularity of the integrand and the integration dimension. The integrand variation and integration dimension dictate the classical upper bound for {the} QMC {method}, known as the Koksma{–}Hlawka inequality~\cite{niederreiter1992random}.

Concerns about the quality of point sets arise as the integration dimension increases~\cite{morokoff1994quasi, wang2008low, kuo2016practical, gerber2015sequential}. For instance, the Halton sequence may exhibit pathological behavior in certain dimensions{; thus,} researchers have proposed remedies~\cite{dong2022dependence, owen2017randomized}.

Despite these challenges, the QMC method has been successful in relatively high dimensions due to the ``low effective dimension''~\cite{wang2003effective,wang2005high}. This notion first arose from the analysis of variance (ANOVA) decomposition in~\cite{caflisch1997valuation}. Several studies have aimed to apply various approaches to minimize the effective dimension, particularly in financial applications~\cite{Bayer2016SmoothingTP, Bayer2020NumericalSA, wang2009dimension}. {Nevertheless}, an equivalence theorem in~\cite{wang2011quasi}, an extension of the worst-case integration error analysis in~\cite{papageorgiou2002brownian}, stated that no decomposition method (includes Brownian {b}ridge, {p}rincipal {c}omponent {a}nalysis, etc.) is consistently superior to other methods in terms of {different} payoff functions. Nonetheless, considering the {exact forms of the payoff function}, the {authors of}~\cite{imai2006general} introduced the linear transformation method, exploit{ing} the nonuniqueness of the covariance matrix decomposition {to find the most important dimensions}. The design of the transformation matrix's columns optimizes the variance contribution of each random variable. 

In addition to the effective dimension, the regularity of the integrand also affects the QMC integration. In finance applications, option pricing problems can sometimes involve discontinuous functions. The orthogonal transformation was proposed in~\cite{wang2013pricing} to transform the discontinuity, involv{ing} a linear combination of the random variables, to become parallel to the axes. The {authors of}~\cite{imai2014pricing} discussed the connection between orthogonal and linear transformation methods. The authors in~\cite{he2015convergence} characterized the ``QMC-friendly'' axis-parallel discontinuity and explicitly derived the convergence rate for integrands with such discontinuities, supported by numerical experiments. {Although the optimal QMC convergence rates are not recovered, the superiorities to the MC methods are shown.}

Conditioning is a well-known approach to reduce variance, but it can also improve the smoothness~\cite{griewank2018high,weng2017efficient,Bayer2016SmoothingTP,bayer2023numerical}. {In}~\cite{Griebel2013TheSE}, {the authors demonstrated that}, under certain conditions, all terms from the ANOVA decomposition except the one {with} the highest order have infinite smoothness. {In addition, in}~\cite{gilbert2022preintegration}, the conditions under which preintegration works {were provided}. In~\cite{Liu2022PreintegrationVA}, the authors proposed preintegrat{ing} in a subspace comprising a linear combination of random input variables. {Moreover, previous studies}~\cite{Jin2006ReclaimingQ,Bayer2022OptimalDW} exploited the regularity with Fourier transformation when the original function exhibits non-smoothness. 

{A partial differential equation} (PDE) with random coefficients is another application {in which the} QMC method accelerates the convergence of computing the statistical properties of a quantity of interest (QoI), often taking the form of a functional of the PDE solution. In~\cite{Sloan1998WhenAQ}, the authors proved {that} the QMC worst-case error is independent of the dimension if the integrand belongs to a class of weighted Sobolev spaces. { Another study}~\cite{nichols2014fast} considered the weights of the form ``product and order dependent'' to minimize the worst-case error and provide a fast method to design the lattices according to the weights. The dimension-independent worst-case QMC errors of elliptic PDEs with affine uniform and lognormal coefficients {were} analyzed in~\cite{Kuo2012QuasiMonteCF} and~\cite{graham2015quasi}. In the latter case, the space is equipped with additional weight functions. {Moreover, in}~\cite{herrmann2019qmc,kazashi2019quasi}{, the authors} considered the ``product weights'' in the weighted space and lattice rule design. More recent studies~\cite{goda2022construction} and~\cite{pan2023super} have considered the median-of-the-mean QMC estimator using lattice rule (without specific design) and digital sequence, respectively. 

Despite the careful analysis of the asymptotic convergence rates in the literature, the ideal QMC convergence rate $\mathcal{O}(n^{-1})$ is not always observed in practice. In this work, we aim to study the nonasymptotic convergence behavior of the QMC method and explain the often observed suboptimal convergence rates. 

{Specifically}, certain integration problems involve integrand domains different from $[0, 1]$, such as the integration with respect to {(w.r.t)} the Gaussian measure on $\mathbb{R}$. The inverse cumulative distribution function (CDF) transformation is applied for compatibility with QMC methods. However, this transformation often introduces singularities at the boundaries. {The work~\cite{Owen06} characterized the boundary singularity with the boundary growth rate and made connections with the asymptotic QMC convergence rate. Particularly, the author provided examples of integrands which blow up at boundaries but still lead to optimal QMC convergence rate.} Building on this work, we {aim} to analyze the nonasymptotic QMC convergence rate for some examples involving lognormal random variables to explain the observed suboptimal rates. {Last, the importance sampling (IS) is a well-known method for variance reduction.} We also aim to discover the extent to which IS with certain proposal distributions improves the convergence rate. 

Some recent publications have delved into the realm of QMC methods for potentially unbounded integrands. One study~\cite{gobet2022mean} combined the robust mean estimator with QMC sampling. In parallel, the work~\cite{he2022error} and~\cite{ouyang2023quasi} studied the QMC method with IS. While these studies provided valuable insights into the asymptotic convergence rate, they did not address the nonasymptotic behaviors frequently observed in real-world scenarios. In contrast, {our work contributes to developing} a {convergence rate} model for QMC methods with a finite number of samples. This innovative approach offers potential pathways to enhance the practically-observed convergence rate.

The paper is organized as follows. Section~\ref{sec:non_asymptotic_convergence_rate_QMC} discusses the nonasymptotic convergence rate. Next, Section~\ref{sec:integrand_infty} presents the convergence rate analysis for two examples: the expectation of a lognormal random variable and elliptic PDEs with lognormal coefficients. {Then,} Section~\ref{sec:importance_sampling} evaluates the effects of two kinds of IS distributions. Section~\ref{sec:numerical} details the numerical results. Finally, Section~\ref{sec:concl} presents the conclusions. 
\section{Nonasymptotic convergence rate for the randomized QMC method}
\label{sec:non_asymptotic_convergence_rate_QMC}
This section follows the proofs in~\cite{Owen06} and accordingly modifies them to establish the nonasymptotic results. First, we introduce the notation. We are interested in the following integration problem:
\begin{align}
	I(g) = \int_{[0, 1]^s} g(\mathbf{t}) d\mathbf{t},
	\label{eq:integration_problem}
\end{align}
where $g: [0, 1]^s \to \mathbb{R}$. We let $\mathcal{P} = \{\mathbf{t}_1, \mathbf{t}_2, \dotsc, \mathbf{t}_{n}\}$ be a point set in $[0, 1]^s$. 

The QMC estimator for the integrand $g$ is given by
\begin{align}
	\hat{I}_n(g)= \frac{1}{n} \sum_{i=1}^n g(\mathbf{t}_i),
	\label{eq:qmc_estimation}
\end{align}
where the $\{\mathbf{t}_i\}_{i=1}^n$ is a predesigned deterministic low-discrepancy sequence~\cite{niederreiter1992random,dick2010digital}. A notion to describe how well the points in $\mathcal{P}$ are uniformly distributed is the star discrepancy $\Delta_{\mathcal{P}}^*$, given by
\begin{align}
\Delta_{\mathcal{P}}^* := \sup_{\mathbf{x} \in [0, 1]^s} \left\lvert \Delta_n (\mathbf{x};\mathcal{P}) \right\rvert,
\end{align}
where $\Delta_n (\mathbf{x};\mathcal{P}) = \frac{1}{n} \sum_{j=1}^n \mathbbm{1}_{\mathbf{t}_j \in [0, \mathbf{x})} - \prod_{k=1}^s x_k$, which is the difference of the measure of $[0, \mathbf{x})$ and the proportion of points that belong to this set, where $[0, \mathbf{x})$ denotes the tensor product of each dimension, i.e. $[0, \mathbf{x}) = \prod_{k=1}^{s} [0, x_k)$. We expect a small difference for an evenly distributed point set. For some low-discrepancy sequences with fixed length $n$, we have
\begin{align}
\Delta_{\mathcal{P}}^* = \mathcal{O}(n^{-1} (\log n)^s), 
\label{eq:star_discrepancy_rate}
\end{align}
where the exponent of the log term becomes $s-1$ for an infinite sequence~\cite{niederreiter1992random,dick2010digital}. 

In this work we consider functions with a continuous mixed first-order derivative. The variation in the Hardy–Krause sense can be computed as follows:
\begin{align}
V_{HK}(g) = \sum_{\emptyset \neq \mathfrak{u} \subseteq 1:s}  \int_{[0, 1]^{\abs{\mathfrak{u}}}} \left\lvert {\partial^{\mathfrak{u}}g(\mathbf{t^{\mathfrak{u}}} \colon \bm{1^{-\mathfrak{u}}} ) } \right\rvert d \mathbf{t^{\mathfrak{u}}},
\end{align}
where $\abs{\mathfrak{u}}$ is the cardinality of the set $\mathfrak{u}$, $\mathbf{y} = {\mathbf{t^{\mathfrak{u}}} \colon \mathbf{1^{-\mathfrak{u}}} } \in [0, 1]^s$ denotes a point in $[0,1]^s$ with ${y^j} = {t^j}$ for $j \in \mathfrak{u}$, and $y^j = 1$ otherwise. For the set $\mathfrak{u} = \{\mathfrak{u}_1, \dotsc, \mathfrak{u}_{\abs{\mathfrak{u}}}\} \subseteq \{1, \dotsc, s\}$, the mixed derivative $\partial^{\mathfrak{u}} g$ is explicitly given by
\begin{equation*}
	\partial^{\mathfrak{u}} g(\mathbf{t}) = \frac{\partial^{\abs{\mathfrak{u}}} g}{\partial \mathbf{{t}}_{\mathfrak{u}}} (\mathbf{{t}}) = \frac{\partial}{\partial t_{\mathfrak{u}_{\abs{\mathfrak{u}}}}} \cdots \frac{\partial}{\partial t_{\mathfrak{u}_1}} g(\mathbf{{t}})
\end{equation*}
where the continuity ensures that the order of differentiation can be switched while maintaining derivative invariance.

The Koksma–Hlawka inequality provides an error estimate for the QMC method, given by 
\begin{align}
\left\lvert I(g) - \hat{I}_{n}(g) \right\rvert \leq V_{HK}(g) \Delta_{\mathcal{P}}^*,
\label{eq:koksma_hlawka_inequality}
\end{align}
where $V_{HK}(g)$ is the variation of $g$ in the Hardy–Krause sense.

However, the use of a deterministic point set yields biased results. Randomization techniques have been introduced to address this {problem}, giving rise to the randomized QMC (RQMC) unbiased estimator~\cite{lecuyer2018randomized}:
\begin{align}
\hat{I}^{r}_{n} (g) = \frac{1}{n} \sum_{i=1}^n g(\mathbf{t}_i \oplus \boldsymbol{\Delta}) \coloneqq \frac{1}{n} \sum_{i=1}^n g(\tilde{\mathbf{t}}_i ),
\end{align}
where $\tilde{\mathbf{t}}_i \coloneqq \mathbf{t}_i \oplus \boldsymbol{\Delta}$ denotes the $i$th randomized QMC quadrature, $\mathbf{t}_i$ is the $i$th deterministic QMC quadrature point, $\boldsymbol{\Delta}$ represents a randomization and typically can be modelled as a random variable from uniform distribution ${U}[0, 1]^s$, and $\oplus$ denotes the randomization operation. An example of such an operation is the random shift, where $\boldsymbol{\Delta}$ is a random variable from uniform distribution ${U}[0, 1]^s$ and $\mathbf{a} \oplus \mathbf{b} = (\mathbf{a} + \mathbf{b}) \textrm{ mod } 1$, with the modulo taken component-wise. This random shift approach is easy to implement and provides an unbiased integral estimator. 

In this work, we consider specifically the problem of evaluating the integral
\begin{align}
	\begin{split}
	I(g) &= \int_{[0, 1]^s} g(\mathbf{t}) d\mathbf{t}\\
	&= \int_{\mathbb{R}^s} g\circ \Phi (\mathbf{y}) \rho (\mathbf{y}) d\bm{y},
	\end{split}
	\label{eq:integrand_g}
\end{align}
where $\circ$ denotes the function composition, $\mathbf{y} = \Phi^{-1} (\mathbf{t})$, $\Phi^{-1} : [0,1]^s \to \mathbb{R}^s$ the inverse CDF of $s$-dimensional {standard} normal distribution, and $\rho$ {represents} its probability density function. Section~\ref{sec:integrand_infty} {presents} two concrete examples of $g$. Before we analyze such an integrand, we study the behavior of uniform random variables. 

\subsection{Some properties of the uniform distribution}

Lemma~\ref{lemma:bound for d-dim uniform RV} introduces a useful property of the uniform distribution. 
\begin{lemma}[A bound for {an} $s$-dimensional uniform random variable]
	Let each $\mathbf{{t}}_i$ in the sequence $\{\mathbf{{t}}_i\}_{i = 1}^{\infty}$ be uniformly distributed over $[0, 1]^s$.  Define $E_n$ as the event $\{\prod_{j=1}^s t_n^j \leq C \cdot n^{-r}\}$ for $r > 1$ and $C>0$. Then, we have
	\begin{align}
	\textnormal{Pr}\left( E_n \enspace i.o.\right) = 0,
	\end{align}
	where $i.o.$ stands for ``infinitely often.''
	\label{lemma:bound for d-dim uniform RV}
\end{lemma}
The sequence $\{\mathbf{t}_i\}$ in Lemma~\ref{lemma:bound for d-dim uniform RV} is not necessarily independent, as is the case for the RQMC method, where the points are desired to exhibit a negative correlation~\cite{Wiart2019OnTD}. Lemma~\ref{lemma:bound for d-dim uniform RV} is a slight modification of Lemma 4.1 in~\cite{Owen06}, where the minimum condition is removed. The proof follows~\cite{Owen06}, except we corrected an upper bound. The proof is provided in the supplementary material~\ref{sec:proof_of_lemma_bound_for_d-dim_uniform_RV}. 

\begin{remark}[Results in earlier literature]
	Lemma 4.1 in~\cite{Owen06} states a stronger result, namely that
	\begin{align}
	\textnormal{Pr}\left( \min_{1\leq i \leq n} \prod_{j=1}^s t_i^j \leq C \cdot n^{-r} \enspace i.o.\right) = 0.
	\label{eq:owen_statement}
	\end{align}
	However, the additional condition 
	\begin{align*}
		\min_{1\leq i \leq n}
	\end{align*}
{leads} to a conclusion that is different from the statement~\eqref{eq:owen_statement}. Specifically, if we assume $\{\mathbf{t}_i\}_{i=1}^n$ are independently and identically distributed (i.i.d.) $U[0, 1]^s$ and {define} the event $F_n = \{ \min_{1\leq i \leq n} \prod_{j=1}^s t_i^j \leq  n^{-r}\}$. We choose $C=1$ for simplicity. Then,
	 \begin{align*}
	 \textnormal{Pr}\left( F_n \enspace i.o. \right) &= \textnormal{Pr}\left( \bigcap_{n=1}^{\infty} \bigcup_{k = n}^{\infty} F_k \right) \nonumber \\
	 &= \lim_{n \to \infty} \textnormal{Pr}\left( \bigcup_{k = n}^{\infty} F_k \right)\\
	 &\geq \lim_{n \to \infty} \textnormal{Pr}\left( F_n \right)\\
	 &= \lim_{n \to \infty} \left( 1 - \prod_{i=1}^n \textnormal{Pr}\left(  \prod_{j=1}^s t_i^j > n^{-r}  \right) \right).
	 \end{align*}
	 For example, when $s = 1$,
	 \begin{align*}
	 	\lim_{n \to \infty} \left( 1 - \prod_{i=1}^n \textnormal{Pr}\left( t_i > n^{-r}  \right) \right) &= \lim_{n \to \infty} \left( 1 - \left( 1 - n^{-r}  \right)^n \right)\\
	 	&= 1 - \exp\left(-\frac{1}{r}\right).
	 \end{align*}
\end{remark}

\begin{remark}[The nonunique constant $C$]
The constant $C$ in Lemma~\ref{lemma:bound for d-dim uniform RV} is not determined uniquely. Nevertheless, we retain this constant $C$ in the estimation model. 
\end{remark}

The event $E_n$, as defined in the above proof, only occurs finitely many times \textit{almost surely}. The origin is bounded away in this sense. {Through} symmetry arguments, we can {similarly show} that all the $2^s$ corners are bounded away from the following hyperbolic set{:}
\begin{align}
K_{n, s} = \left\{\mathbf{t}\in[0,1]^s\mid \prod_{1\leq j \leq s} \min(t_j, 1-t_j) \geq C n^{-1} \right\},
\label{eq:hyperbolic_set}
\end{align}
where we use the boundary case $n^{-1}$ instead of $n^{-r}$ from Lemma~\ref{lemma:bound for d-dim uniform RV}. We also set $Cn^{-1}\leq 1$ to exclude trivial cases. 
\begin{figure}[htbp]
	\centering
	\includegraphics[width=0.48\textwidth]{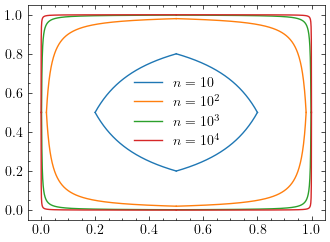}
	\hfill
	\includegraphics[width=0.48\textwidth]{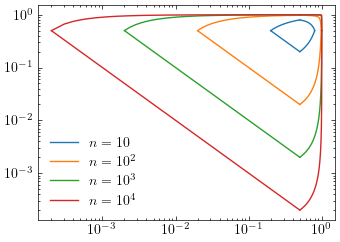}
	\caption{Illustration of the hyperbolic set $K_{n,s}$ introduced in~\eqref{eq:hyperbolic_set} using {the} linear scale (left) and log scale (right) for various $n$, where $C=1$. }
	\label{fig:K_ns}
\end{figure}
From Lemma~\ref{lemma:bound for d-dim uniform RV}, we see that the uniform distribution samples ``diffuse'' to the corner at a certain rate. Figure~\ref{fig:uniform_samples} plots the $n$ samples from {the} uniform distribution $U[0, 1]^2$ and the corresponding reference boundaries $t_1t_2 = n^{-1}$. The samples approach the corner when size increases and a small portion of samples lie outside their corresponding reference boundaries.
\begin{figure}[htbp]
	\centering
	\includegraphics[width = 0.75\textwidth]{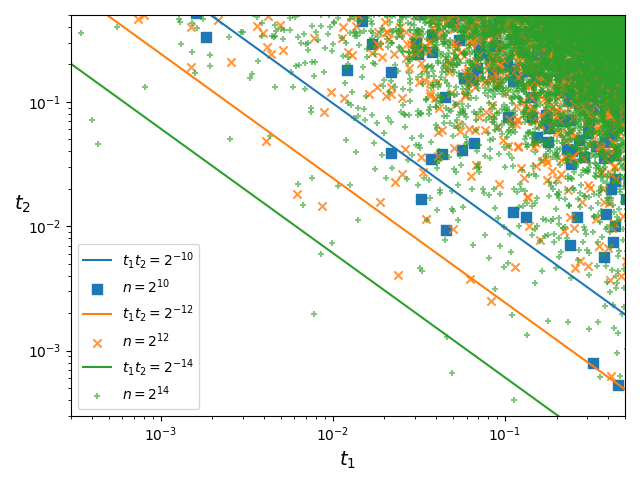}
	\caption{Samples from {a two-dimensional} uniform distribution $U[0,1]^2$ with various sample sizes $n$, and the corresponding reference boundaries $t_1 t_2 =n^{-1}$.}
	\label{fig:uniform_samples}
\end{figure}
\subsection{Integrand with infinite variation}

We are often interested in integrands with infinite variation, for which the Koksma–Hlawka inequality~\eqref{eq:koksma_hlawka_inequality} is not informative. However, cases remain where the QMC methods work. The study~\cite{he2015convergence} considered discontinuities inside the integration domain and proved that the convergence rate is still superior to the Monte Carlo method as long as some axes are parallel to the discontinuity interface. Another study~\cite{Owen06} explored integrands that blow up at the boundary of $[0, 1]^s$ and derived convergence rates {based on a} boundary growth condition. We are interested in the latter case. Following the work~\cite{Owen06}, the set $K_{n,s}$ was introduced to split the integration domain {and the Sobol' low-variation extension~\cite{Sobol1973calculation} $\tilde{g}$ extends $g$ from $K_{n, s}$ to $[0, 1]^s$.} The extension $\tilde{g}$ depends on $n$ and $s$ through the hyperbolic set $K_{n,s}$, but we suppress these dependences for the sake of simpler notations. Though the original name in the literatrue states ``low-variation'', we assure the readers that the extension $\tilde{g}$ indeed has finite variation. For completeness, we provide the details here.

To this end, {some} notations need to be introduced. We denote $1:s \coloneqq \{1, 2, \dotsc, s\}$. For an index set $\mathfrak{u} \subseteq 1:s$, $-\mathfrak{u}$ denotes its complement $1:s \setminus \mathfrak{u}$. The notation $\mathbf{z^{\mathfrak{u}}}:\mathbf{c^{-\mathfrak{u}}}$ is used to denote the point $\mathbf{y} \in [0, 1]^s$, where $y^j = z^j$ for $j \in \mathfrak{u}$ and $y^j = c^j$ for $j \notin \mathfrak{u}$, ``concatenating'' the vectors $\mathbf{z^{\mathfrak{u}}}$ and $\mathbf{c^{\mathfrak{-u}}}$. For simplicity, we assume the derivative of the integrand $\partial^{\mathfrak{u}}g$ exists in $(0, 1)^s$, for $\mathfrak{u} \subseteq 1:s$. Given an anchor point $\mathbf{c} \in K_{n,s}$ {and} using the fundamental theorem of calculus, we obtain
\begin{align}
	g(\mathbf{t}) = g(\mathbf{c}) + \sum_{\mathfrak{u} \neq \emptyset} \int_{[\mathbf{c^{\mathfrak{u}}}, \mathbf{t^{\mathfrak{u}}}]} \partial^{\mathfrak{u}} g(\mathbf{z^{\mathfrak{u}}}:\mathbf{c^{-\mathfrak{u}}})d\bm{z^{\mathfrak{u}}}. 
\end{align}
The Sobol' low-variation extension $\tilde{g}:\mathbb{R}^s \to \mathbb{R}$ is given by
\begin{align}
	\label{eq:sobol_low_variation_extension}
	\tilde{g}(\mathbf{t}) = g(\mathbf{c}) + \sum_{\mathfrak{u} \neq \emptyset} \int_{[\mathbf{c^{\mathfrak{u}}}, \mathbf{t^{\mathfrak{u}}} ] } \mathbbm{1}_{\{ \mathbf{z^{\mathfrak{u}}}:\mathbf{c^{-\mathfrak{u}}} \in K_{n, s}\}} ~ \partial^{\mathfrak{u}} g(\mathbf{z^{\mathfrak{u}}}:\mathbf{c^{-\mathfrak{u}}})d\bm{z^{\mathfrak{u}}}. 
\end{align}

Following~\cite{Owen06, Sobol1973calculation}, {we apply} a three-epsilon argument to {bound} the integration error. Recalling the notations on the exact integration~\eqref{eq:integration_problem} and QMC estimation~\eqref{eq:qmc_estimation}, we bound:
\begin{align}
\begin{split}
\lvert I(g) - \hat{I}^r_n(g) \rvert &\leq \lvert I(g) - I(\tilde{g}) \rvert + \lvert I(\tilde{g}) - \hat{I}^r_n(\tilde{g}) \rvert + \lvert \hat{I}^r_n(\tilde{g}) - \hat{I}^{r}_n(g) \rvert\\
& \leq \int_{[0, 1]^s - K_{n, s}} \lvert g(\mathbf{t}) - \tilde{g}(\mathbf{t}) \rvert d\mathbf{t} + \Delta^*(\tilde{\mathbf{t}}_1, \dotsc, \tilde{\mathbf{t}}_n)V_{HK}(\tilde{g}) + \frac{1}{n}\sum_{i=1}^n \lvert \tilde{g}(\tilde{\mathbf{t}}_i) - g(\tilde{\mathbf{t}}_i) \rvert.
\end{split}
\end{align}
{Observe} that $g$ and $\tilde{g}$ coincide in $K_{n, s}$. Moreover, we have
\begin{align}
\mathbb{E} \left[\frac{1}{n}\sum_{i=1}^n \left\lvert \tilde{g}(\tilde{\mathbf{t}}_i) - g(\tilde{\mathbf{t}}_i) \right\rvert \right] &=  \frac{1}{n}\sum_{i=1}^n \mathbb{E} \left[ \left\lvert \tilde{g}(\tilde{\mathbf{t}}_i) - g(\tilde{\mathbf{t}}_i) \right\rvert \right] \nonumber \\
&= \int_{[0, 1]^s - K_{n, s}} \lvert g(\mathbf{t}) - \tilde{g}(\mathbf{t}) \rvert d\mathbf{t}.
\end{align}
Using the above inequalities, we obtain the following finite upper bound for the RQMC method:
\begin{align}
\begin{split}
\mathbb{E} [ \lvert I(g) - \hat{I}^{r}_n(g) \rvert ] & \leq 2\int_{[0, 1]^s - K_{n, s}} \lvert g(\mathbf{t}) - \tilde{g}(\mathbf{t}) \rvert d\mathbf{t} + \mathbb{E} \left[ \Delta^*(\tilde{\mathbf{t}}_1, \dotsc, \tilde{\mathbf{t}}_n) \right]V_{HK}(\tilde{g}).
\end{split}
\label{eq:expected_rqmc_upper_bound}
\end{align}

Notice that the above bound~\eqref{eq:expected_rqmc_upper_bound} depends on the value of $C$, which is used to construct the set $K_{n,s}$ in Equation~\eqref{eq:hyperbolic_set}. When $C = 0$, this becomes the classical Koksma–Hlawka inequality, and when $C = \frac{n}{2^s}$, the hyperbolic set $K_{n,s}$ vanishes and the right hand side becomes $2 \int_{[0, 1]^s } \lvert g (\mathbf{t}) \rvert d \mathbf{t}$. For $0< C < \frac{n}{2^s}$, the {asymptotic} upper bound also reduces to the Koksma–Hlawka inequality, as $n \to \infty$. 

To examine the behavior of the difference $g - \tilde{g}$ and the variation $V_{HK}(\tilde{g})$ as the set $K_{n,s}$ grows when $n$ increases, we assume the following boundary growth condition (in the nonasymptotic case) outside the set $ K_{n, s}$.
\begin{assumption}[Nonasymptotic boundary growth condition]
	For the integrand $g$, we assume the following boundary growth condition. For each fixed $n \in \mathbb{N}$, $\forall \mathfrak{u} \subseteq 1:s$, there exist $0 < A_j < \frac{1}{2}, \forall j \in 1:s$:
	\begin{align}
	\lvert \partial^{\mathfrak{u}} g(\mathbf{t}) \rvert \leq B(\mathbf{t}_c) \prod_{j=1}^s \min{(t_j, 1 - t_j)}^{-A_j(t_c^j) - \mathbbm{1}_{j \in \mathfrak{u}}} \quad \mathbf{t} \in [0, 1]^s, \quad \mathbf{t}_c \in \partial K_{n,s},
	\label{eq:boundary_growth_condition}
	\end{align}
	where we assume $B(\mathbf{t}_c) < +\infty$. 
	\label{assump:boundary_growth_condition}
\end{assumption}

In the rest of this paper, we will only consider one orthant of the unit cube, i.e.,  $\mathbf{t} \in [0, \frac{1}{2}]^s$ and $\mathbf{t}_c \in \partial K_{n,s} \cap [0, \frac{1}{2}]^s$, as other cases can be similarly derived by symmetry. Throughout this work, the inequalities for vectors are assumed to hold componentwise. The dependence $A_j = A_j(\mathbf{t}_c)$ is crucial in the following nonasymptotic analysis. For now, we consider the example of an anisotropic integrand. 
\begin{example}[An anisotropic integrand]
	\begin{equation}
		g(\mathbf{t}) = \exp(2\Phi^{-1}(1 - t_1) + \Phi^{-1}(1 - t_2)) \quad \mathbf{t} \in [0, 1]^2. 
			\label{eq:anisotropic_integrand}
	\end{equation}
\end{example}
As is often the case, the function~\eqref{eq:anisotropic_integrand} in our example is not isotropic w.r.t. the coordinates. 
The value of $\lvert \partial^{1:s} g(\mathbf{t}) \rvert$ may not be constant for $\mathbf{t} \in \partial K_{n,s}$. 
%
By Assumption~\ref{assump:boundary_growth_condition}, we have
\begin{equation}
	\label{eq:upper_bound_bgc}
	\begin{split}
		\abs*{\partial^{1:s} g(\mathbf{t})} &\leq B(\mathbf{t}_c) \prod_{j=1}^s \min({t_j}, 1 - t_j)^{-A_j(t_c^j) - 1}.
	\end{split}
\end{equation}
In the following we aim to find {an} upper bound of the following term:
\begin{align*}
	B(\mathbf{t}_c) \prod_{j=1}^s \min({t_j}, 1 - t_j)^{-A_j(t_c^j)},  ~ \textrm{where}~ \mathbf{t}_c \in \partial K_{n,s}
\end{align*}
in dimension $s > 1$. Without loss of generality, we assume the maximizer is found when $\mathbf{t}_c \in [0, \frac{1}{2}]^s$ due to the symmetry of the domain. This leads to the following optimization problem:
\begin{align}
	\begin{split}
		\max_{\mathbf{t}_c \in \partial K_{n,s}}&~ \log B(\mathbf{t}_c) +  \sum_{j=1}^s -A_j \log t_c^j\\
		s.t.&~ \sum_{j=1}^s \log t_c^j = \log \delta (n),
	\end{split}
	\label{eq:optimization_aj}
\end{align}
for a given $\delta(n) = Cn^{-1} > 0$. 
We define $\mathbf{t}_c^*$ as the optimizer of~\eqref{eq:optimization_aj} and
\begin{equation}
	A_j^* \coloneqq A_j(\mathbf{t}_c^*), B^{*}_{n, s} \coloneqq B(\mathbf{t}_c^*)
	\label{eq:a_star}
\end{equation}
for $j = 1, \dotsc, s$. In one-dimensional case, the point $\mathbf{t}_c^*$ is automatically given by $Cn^{-1}$. 
Next we derive the bound for the difference between the integrand and its Sobol' low-variation extension.

\begin{lemma}[A bound for the difference between the integrand and its Sobol' low-variation extension]
Let $\mathbf{t} \in [0, 1]^s$. For integrand $g$ satisfying the boundary growth condition~\eqref{eq:boundary_growth_condition}, we have
	\begin{align}
		\lvert g(\mathbf{t}) - \tilde{g}(\mathbf{t}) \rvert \leq \tilde{B}_{n, s}
		 \prod_{j=1}^s \min{(t_j, 1 - t_j)}^{-A_j^*},
	\end{align}
	where $A_j^*$ is determined by the optimization problem~\eqref{eq:optimization_aj}, and $\tilde{B}_{n, s}  = 2^s {B^{*}_{n, s} }\prod_{j=1}^s (1 + \frac{1}{A_j^{*}})$.
	\label{lemma:difference_integrand_sobol_low_variation_extension}
\end{lemma}
\begin{proof}
We first consider the case when $\mathbf{t} \in [0, \frac{1}{2}]^s$. Notice that from~\eqref{eq:boundary_growth_condition}, ~\eqref{eq:optimization_aj} and~\eqref{eq:a_star}, we have
\begin{equation}
	\abs*{\partial^{\mathfrak{u}} g(\mathbf{t})} \leq B^{*}_{n, s} \prod_{j=1}^s {t_j}^{-A_j^{*} - \mathbbm{1}_{j \in \mathfrak{u}}}, \quad \mathfrak{u} \subseteq 1:s.
\end{equation}
Following the proof of Lemma 5.1 in~\cite{Owen06}, we have
\begin{equation}
	\label{eq:g_minus_g_tilde}
	\begin{split}
	\lvert g(\mathbf{t}) - \tilde{g}(\mathbf{t}) \rvert &\leq \sum_{\substack{\mathfrak{u} \subseteq 1:s \\ \mathfrak{u} \neq \varnothing}} \int_{[\mathbf{t}^{\mathfrak{u}}, \left(\bm{\frac{1}{2}}\right)^{\mathfrak{u}}] } \mathbbm{1}_{\mathbf{z}^{\mathfrak{u}}:\bm{\frac{1}{2}}^{-\mathfrak{u} } \notin K_{n, s} } \partial^{\mathfrak{u}} g(\mathbf{z}^{\mathfrak{u}} : \bm{\frac{1}{2}}^{-\mathfrak{u}} ) d\mathbf{z}^{\mathfrak{u}}  \\
	&\leq 2^s B^{*}_{n, s} \sum_{\substack{\mathfrak{u} \subseteq 1:s \\ \mathfrak{u} \neq \varnothing}} \int_{[\mathbf{t}^{\mathfrak{u}}, \left(\bm{\frac{1}{2}}\right)^{\mathfrak{u}}] }  \prod_{j \in \mathfrak{u}} {z_j}^{-A_j^{*} - {1}} d\mathbf{z}^{\mathfrak{u}} \\
	&\leq 2^s B^{*}_{n, s} \sum_{\substack{\mathfrak{u} \subseteq 1:s \\ \mathfrak{u} \neq \varnothing}} \prod_{j \in \mathfrak{u}} \frac{t_j^{-A_j^{*}}}{A_j}\\
	&= 2^s B^{*}_{n, s} \left( \prod_{j=1}^s \left( 1 + \frac{t_j^{-A_j^{*}}}{A_j} \right) - 1 \right)\\
	&\leq 2^s B^{*}_{n, s} \prod_{j=1}^s \left( 1 + \frac{1}{A_j^{*}} \right) t_j^{-A_j^{*}}. 
	\end{split}
\end{equation}
The rest of the cases when $\mathbf{t} \notin [0, \frac{1}{2}]^s$ can be similarly derived by symmetry. This concludes the proof.

\end{proof}

Now we define the set $G(\delta) \coloneqq \{\mathbf{t} \in [0, 1]^s: \prod_{j=1}^s t_j \geq \delta\}$ for $\delta > 0$ and consider the integration in $G(\delta)$ and $[0, 1]^s - G(\delta)$. 
\begin{lemma}[Integral bound for the difference between the integrand and its Sobol' low-variation extension]
	When $\lvert g(\mathbf{t}) - \tilde{g}(\mathbf{t}) \rvert \leq \tilde{B}_{n,s}  \prod_{j=1}^s {t_j}^{-A_j}$, $\mathbf{t} \in [0, 1]^s$, and $A_j < 1$ are distinct for $j = 1, \dotsc, s$, we have 
	\begin{equation}
			\int_{[0, 1]^s - G(\delta)} \lvert g(\mathbf{t}) - \tilde{g}(\mathbf{t})  \rvert d\mathbf{t} \leq \tilde{B}_{n,s} \sum_{j=1}^s \frac{\delta^{1-A_j}}{(A_j - 1) \phi^{\prime}(A_j) },
	\end{equation}
	where $\phi$ is the polynomial $\phi(x) \coloneqq \prod_{j=1}^s (x - {A_j})$. And when $A_1 = \cdots = A_s = A < 1$, we have
	\begin{equation}
		\int_{[0, 1]^s - G(\delta)} \lvert g(\mathbf{t}) - \tilde{g}(\mathbf{t})  \rvert d\mathbf{t} \leq - \tilde{B}_{n,s} \sum_{k=1}^{s} \frac{\delta^{1-A}}{(1-A)^{k}} \cdot \frac{\log^{s-k} \delta}{(s-k)!}.
	\end{equation}
	\label{lemma:integration_difference_integrand_sobol_low_variation_extension}
\end{lemma}
Notice that the first part of Lemma~\ref{lemma:integration_difference_integrand_sobol_low_variation_extension} addresses a typo in Lemma 5.2 in~\cite{Owen06}, which cites Lemma 3 in~\cite{Sobol1973calculation}. The original Lemma is stated as follows:
\begin{lemma}[Sobol', 1973]
	\label{lemma:sobol_1973}
	For distinct $A_1, \dotsc, A_j < 1$, we have
	\begin{equation}
		\int_{G(c)} \prod_{j=1}^s t_j^{-A_j} d\bm{t} = \frac{1}{\phi(1)} + \sum_{j=1}^s \frac{c^{1 - A_j}}{(A_j - 1)\phi^{\prime} (A_j)},
 	\end{equation}
	where $\phi$ is defined the same as in Lemma~\ref{lemma:integration_difference_integrand_sobol_low_variation_extension}. \footnote{However, to the best of authors' knowledge, $G(\delta)$ is not exactly defined as above in~\cite{Sobol1973calculation}.}
\end{lemma}
The exact derivation of Lemma~\ref{lemma:integration_difference_integrand_sobol_low_variation_extension} and~\ref{lemma:sobol_1973} are provided in the supplementary material~\ref{sec:proof_of_lemma_integration_difference_integrand_sobol_low_variation_extension}. Now we have,
\begin{equation}
	\begin{split}
		\int_{[0, \frac{1}{2}]^s - K_{n,s}} \lvert g(\mathbf{t}) - \tilde{g}(\mathbf{t})  \rvert d\mathbf{t} &\leq \int_{[0, \frac{1}{2}]^s - G(Cn^{-1})} \tilde{B}_{n,s}  \prod_{j=1}^s {t_j}^{-A_j^{*}} d\mathbf{t}.
	\end{split}
\end{equation}
and by symmetricity, we have
\begin{equation}
	\label{eq:integration_err_g_tilde_1_Kns}
	\begin{split}
	\int_{[0, 1]^s - K_{n,s}} \lvert g(\mathbf{t}) - \tilde{g}(\mathbf{t})  \rvert d\mathbf{t} &\leq 2^s \int_{[0, 1]^s - G(Cn^{-1})} \tilde{B}_{n,s}  \prod_{j=1}^s {t_j}^{-A_j^{*}} d\mathbf{t}\\
	&\leq 2^s \tilde{B}_{n,s} \sum_{j=1}^s \frac{(Cn^{-1})^{1-A_j^{*}}}{(A_j^{*} - 1) \phi^{\prime}(A_j^{*}) }.
	\end{split}
\end{equation}



In the following we derive an upper bound for the Hardy--Krause variation of the Sobol' low-variation extension $\tilde{g}$. Following~\eqref{eq:sobol_low_variation_extension} and fix the anchor point $\mathbf{c} = (\frac{1}{2}, \dotsc, \frac{1}{2})$, we have
\begin{equation}
	\begin{split}
		V_{HK}(\tilde{g}) &\leq \sum_{\substack{\mathfrak{u} \subseteq 1:s \\ \mathfrak{u} \neq \varnothing}} \int_{[0, 1]^{\abs{\mathfrak{u}}}} \mathbbm{1}_{\mathbf{t} \in K_{n, s}} \abs*{\partial^{\mathfrak{u}} g\left(\mathbf{t}^{\mathfrak{u}}: \bm{\frac{1}{2}}^{-\mathfrak{u}} \right)} d\mathbf{t}^{\mathfrak{u}}.
	\end{split}
\end{equation}
Notice that we have, 
\begin{align}
	\abs*{\partial^{\mathfrak{u}} g\left(\mathbf{t}^{\mathfrak{u}}: \bm{\frac{1}{2}}^{-\mathfrak{u}} \right)} \leq B^{*}_{n,s} \prod_{j \in \mathfrak{u}} \min(t_j, 1 - t_j)^{-A_j^* - 1} \cdot \left( \frac{1}{2}\right)^{\abs{-\mathfrak{u}} },
	\label{eq:variation_g_tilde}
\end{align}
where $A_j^{*}$ is determined by the optimization problem~\eqref{eq:optimization_aj}. Due to the symmetry, we have
\begin{equation}
	\begin{split}
	& \int_{[0, 1]^{\abs{\mathfrak{u}}}} \mathbbm{1}_{\mathbf{t} \in K_{n, s}} \abs*{\partial^{\mathfrak{u}} g\left(\mathbf{t}^{\mathfrak{u}}: \bm{\frac{1}{2}}^{-\mathfrak{u}} \right)} d\mathbf{t}^{\mathfrak{u}}\\
	 &\leq \int_{[0, 1]^{\abs{\mathfrak{u}}}} \mathbbm{1}_{\mathbf{t} \in K_{n, s}} B^{*}_{n,s} \prod_{j \in \mathfrak{u}} \min(t_j, 1 - t_j)^{-A_j^* - 1} \cdot \left( \frac{1}{2}\right)^{\abs{-\mathfrak{u}} } d\mathbf{t}^{\mathfrak{u}}\\
	&\leq 2^{\abs{\mathfrak{u}}} \int_{\prod_{j \in \mathfrak{u}} t_j \geq Cn^{-1} }   B^{*}_{n,s} \prod_{j \in \mathfrak{u}}  t_j^{-A_j^* - 1} \cdot \left( \frac{1}{2}\right)^{\abs{-\mathfrak{u}} } d\mathbf{t}^{\mathfrak{u}}. 
	\end{split}
\end{equation}
Following~\cite{Owen06}, assuming the distinct values for $A_1^{*}, \dotsc, A_s^{*} < 1$, we denote the index $m(\mathfrak{u}) \coloneqq \max_{j \in \mathfrak{u}} A_j$ and the set $\mathfrak{u}_{-} \coloneqq \mathfrak{u} \setminus m(\mathfrak{u})$. We have
\begin{equation}
	\label{eq:hk_variation_g_tilde}
	\begin{split}
		V_{HK}(\tilde{g}) &\leq \sum_{\substack{\mathfrak{u} \subseteq 1:s \\ \mathfrak{u} \neq \varnothing}} \int_{[0, 1]^{\abs{\mathfrak{u}}}} \mathbbm{1}_{\mathbf{t} \in K_{n, s}} \abs*{\partial^{\mathfrak{u}} g\left(\mathbf{t}^{\mathfrak{u}}: \bm{\frac{1}{2}}^{-\mathfrak{u}} \right)} d\mathbf{t}^{\mathfrak{u}}\\
		 &\leq \sum_{\substack{\mathfrak{u} \subseteq 1:s \\ \mathfrak{u} \neq \varnothing}} 2^{\abs{\mathfrak{u}}} \int_{\prod_{j \in \mathfrak{u}} t_j \geq Cn^{-1} }  B^{*}_{n,s} \prod_{j \in \mathfrak{u}} t_j^{-A_j^* - 1} \cdot \left( \frac{1}{2}\right)^{\abs{-\mathfrak{u}} } d\mathbf{t}^{\mathfrak{u}}\\
	&\leq B^{*}_{n,s} \sum_{\substack{\mathfrak{u} \subseteq 1:s \\ \mathfrak{u} \neq \varnothing}} 2^{\abs{\mathfrak{u}}} \left( \frac{1}{2}\right)^{\abs{-\mathfrak{u}} } \int_{\prod_{j \in \mathfrak{u}} t_j \geq Cn^{-1} } \prod_{j \in \mathfrak{u}} t_j^{-A_j^* - 1}  d\mathbf{t}^{\mathfrak{u}}\\
	&\leq B^{*}_{n, s} \sum_{\substack{\mathfrak{u} \subseteq 1:s \\ \mathfrak{u} \neq \varnothing}} 2^{\abs{\mathfrak{u}}} \left( \frac{1}{2}\right)^{\abs{-\mathfrak{u}} } \int_{[0, 1]^{\abs{\mathfrak{u}_{-} }}} \frac{(Cn^{-1})^{-A_{m(\mathfrak{u})}^{*} }}{A_{m(\mathfrak{u})}^{*}} \prod_{j \in \mathfrak{u}_{-}} t_j^{A_{m(\mathfrak{u})}^{*} -A_j^* - 1}  d\mathbf{t}^{\mathfrak{u}_{-}}\\
	&= B^{*}_{n, s} \sum_{\substack{\mathfrak{u} \subseteq 1:s \\ \mathfrak{u} \neq \varnothing}} 2^{\abs{\mathfrak{u}}} \left( \frac{1}{2}\right)^{\abs{-\mathfrak{u}} } \frac{(Cn^{-1})^{-A_{m(\mathfrak{u})}^{*} }}{A_{m(\mathfrak{u})}^{*}} \prod_{j \in \mathfrak{u}_{-}} \frac{1}{A_{m(\mathfrak{u})}^{*} -A_j^*}\\
	&\leq B^{*}_{n, s} \left( \frac{5}{2} \right)^s \sum_{\substack{\mathfrak{u} \subseteq 1:s \\ \mathfrak{u} \neq \varnothing}} \frac{1}{A_{m(\mathfrak{u})}^{*}} \prod_{j \in \mathfrak{u}_{-}} \frac{1}{A_{m(\mathfrak{u})}^{*} -A_j^*} (C^{-1} n)^{\max_{j} A_{j}^{*} }. 
	\end{split}
\end{equation}
Following the remark in~\cite{Owen06}, when the values of $A_1^{*}, \dotsc, A_s^{*}$ are not distinct, we can slightly increase some of the values to make them distinct. Specifically, we can perturb the values such that $A_{i}^{*} - A_{j}^{*} \geq \bar{\delta}$ for $i \neq j$, $\bar{\delta} > 0$. 
{Now, we have obtained the ingredients to present the nonasymptotic upper bound for the RQMC upper bound. }
\begin{theorem}[Nonasymptotic RQMC error estimate] For a continuous integrand $g$ on $[0, 1]^s$, whose mixed first-order derivative exists, and assume for the randomized low-discrepancy sequence, $\mathbb{E} \left[ \Delta^*(\tilde{\mathbf{t}}_1, \dotsc, \tilde{\mathbf{t}}_n) \right] \leq C_{\epsilon, s} n^{-1 + \epsilon} $ for $\epsilon > 0$. $C_{\epsilon, s} < +\infty$ depends on $\epsilon, s$ but not $n$. Assuming $A_j^{*}$, $j = 1, \dotsc, s$ from~\eqref{eq:optimization_aj} takes distinct values, then the expected integration error for the RQMC method with a quadrature size of $n$ satisfies
	\begin{align}
		\mathbb{E} [ \lvert I(g) - \hat{I}_n^r(g) \rvert ] \leq C_1(n) n^{-1 + \max_j A_j^*} + C_2(n) n^{-1 + \epsilon + \max_j A_j^*}
		\label{eq:QMC_convergence_model}
	\end{align}
where $\epsilon > 0$, and $C_1(n) = 2^{s+1} \tilde{B}_{n,s} \sum_{j=1}^s \frac{C^{1-A_j^{*}}}{(A_j^{*} - 1) \phi^{\prime}(A_j^{*}) } $ and \\ $C_2(n) = C_{\epsilon, s} B^{*}_{n, s} \left( \frac{5}{2} \right)^s \sum_{\substack{\mathfrak{u} \subseteq 1:s \\ \mathfrak{u} \neq \varnothing}} \frac{1}{A_{m(\mathfrak{u})}^{*}} \prod_{j \in \mathfrak{u}_{-}} \frac{1}{A_{m(\mathfrak{u})}^{*} -A_j^*} (C^{-1})^{\max_{j} A_{j}^{*} }$. 
	\label{thm:nonasymptotic_QMC_error_estimate}
\end{theorem}

\begin{proof}
	{Substituting Equation~\eqref{eq:integration_err_g_tilde_1_Kns} and~\eqref{eq:hk_variation_g_tilde} into~\eqref{eq:expected_rqmc_upper_bound}} leads to the conclusion.
\end{proof}
The assumption $\mathbb{E} \left[ \Delta^*(\tilde{\mathbf{t}}_1, \dotsc, \tilde{\mathbf{t}}_n) \right] \leq C_{\epsilon, s} n^{-1 + \epsilon} $ is justified in~\cite{Owen06}, for a $(t, s)$-sequence with Owen or Matou\v sek scrambling. The error estimate is also valid for other randomized low-discrepancy sequences when the star-discrepancy assumption is satisfied. Theorem~\ref{thm:nonasymptotic_QMC_error_estimate} is the contribution of this work to the nonasymptotic RQMC error estimate model. In the next section, we will analyze the nonasymptotic error bound for two specific examples.
\begin{remark}[Nonasymptotics rather than singularity]
	In some cases, for instances, the two examples considered in the next Section~\ref{sec:integrand_infty}, both integrands exhibit singularities at the boundary. Nevertheless, the value of $A_j^*$ from~\eqref{eq:a_star} converges to 0 as $n\to \infty$. In this situation, the convergence rate for the RQMC method becomes the optimal asymptotic rate $\mathcal{O}(n^{-1+\epsilon})$ for $\epsilon>0$, as derived in~\cite{Owen06}, in the presence of the singularities and hence infinite variations. However, a suboptimal convergence rate may be observed for a finite sample size $n$, as indicated by Theorem~\ref{thm:nonasymptotic_QMC_error_estimate}. 
\end{remark}

\begin{example}
	\label{example:2}
	We consider an example where given $n_1$, the convergence rate for $n > n_1$ can be explicitly found. Let the error have the following upper bound:
	\begin{equation}
		\exp(\sqrt{\mathcal{C} \log n}) n^{-1 + \sqrt{\frac{\mathcal{C}}{\log n}}}
	\end{equation}
	for a constant $\mathcal{C} > 0$. Notice that this upper bound is similar to the ones in the two examples that we will consider in the next section, in the sense that when $n$ increases, the term $\exp(\sqrt{\mathcal{C} \log n})$ increases while the exponent term $-1 + \sqrt{\frac{\mathcal{C}}{\log n}}$ decreases. Specifically, $\exp(\sqrt{\mathcal{C} \log n}) = o(n^{\alpha})$ for $\alpha > 0$, $n \to \infty$. To see this, we have
	\begin{equation}
		\lim_{n \to \infty} \frac{\exp(\sqrt{\mathcal{C} \log n})}{n^{\alpha}} = \lim_{n \to \infty} \exp(\sqrt{\mathcal{C} \log n} - \alpha \log n) = 0.
	\end{equation}
	In the following we try to illustrate the convergence of such example. Let $n > n_1$, we start from the following inequality:
	\begin{equation}
		2 \sqrt{\mathcal{C} \log n} \leq \sqrt{\mathcal{C} \log n_1} + \sqrt{\mathcal{C} \frac{\log^2 n}{\log n_1}}.
	\end{equation}
	We have,
	\begin{equation}
		\begin{split}
			\sqrt{\mathcal{C} \log n} - \sqrt{\mathcal{C} \log n_1} &\leq  \sqrt{\mathcal{C} \frac{\log^2 n}{\log n_1}} - \sqrt{\mathcal{C} \log n}\\
			&= \left( \sqrt{\mathcal{C} \frac{1}{\log n_1}} - \sqrt{\mathcal{C} \frac{1}{\log n}} \right) \log n. 
		\end{split}
	\end{equation}
	We take the exponential of both sides and have
	\begin{equation}
		\exp\left( \sqrt{\mathcal{C} \log n} - \sqrt{\mathcal{C} \log n_1} \right) \leq n^{\sqrt{\mathcal{C} \frac{1}{\log n_1}} - \sqrt{\mathcal{C} \frac{1}{\log n}}}.
	\end{equation}
	Rearranging the terms, we obtain
	\begin{equation}
		\exp\left( \sqrt{\mathcal{C} \log n}  \right) n^{-1 + \sqrt{\mathcal{C} \frac{1}{\log n}} } \leq \exp\left( \sqrt{\mathcal{C} \log n_1}  \right) n^{-1 + \sqrt{\mathcal{C} \frac{1}{\log n_1}} }. 
	\end{equation}
	We can explicitly find the convergence rate $n^{-1 + \sqrt{\mathcal{C} \frac{1}{\log n_1}} }$ for $n > n_1$. 
\end{example}

\section{Two examples of infinite variation integrands}
\label{sec:integrand_infty}


This section delves into two typical integration examples: the expectation estimate of lognormal random variables and the QoIs relating to the solution of elliptic PDEs with lognormal coefficients. These examples shed light on the intricate nature of infinite variation integrands.

\subsection{Lognormal random variable}
\label{subsec:lognormal}
We first analyze the integration problem~\eqref{eq:integration_problem}, where $g = \exp ( \boldsymbol{\sigma}^{T} {\Phi}^{-1})$, $\boldsymbol{\sigma} \in \mathbb{R}^s_{+}$, and ${\Phi}^{-1} : [0, 1]^s \to \mathbb{R}^s$ applies the inverse standard normal distribution CDF component-wise. The partial derivative of $g$ w.r.t. $\mathbf{t}_\mathfrak{u}$ is given by
\begin{align}
\frac{\partial g}{\partial \mathbf{t}_\mathfrak{u}} &= \frac{\partial}{\partial \mathbf{t}_\mathfrak{u}}\exp\left(\sum_{j = 1}^s \sigma_j \Phi^{-1}(t_j)\right) \nonumber\\
&= \exp\left(\sum_{j = 1}^s \sigma_j \Phi^{-1}(t_j)\right) \cdot \prod_{j \in \mathfrak{u}} \sigma_j {\partial^j \Phi^{-1}(t_j)}.
\end{align}
We need to analyze the behavior of $\partial \Phi^{-1}$. The asymptotic approximations for $\Phi^{-1}$ are given as follows: 
	\begin{align}
	\Phi^{-1} (t)=
	\begin{cases}
		-\sqrt{-2\log t} + o(1) & \quad t \to 0\\
		\sqrt{-2\log (1 - t)}+ o(1) & \quad t \to 1,
	\end{cases} 
	\label{eq:approx_phi_0}
	\end{align}
see~\cite{patel1996handbook}, Chapter~3.9, for instance. A nonasymptotic approximation of $\Phi^{-1}$ can be referred to Equation (26.2.23) of~\cite{abramowitz1968handbook}: When $0 < t \leq 1/2$,
\begin{equation}
	\label{eq:approx_phi_21}
	\Phi^{-1}(t) = -\tau + \frac{c_0 + c_1 \tau + c_2 \tau^2}{1 + d_1 \tau + d_2 \tau^2 + d_3 \tau^3} + \bar{\epsilon},
\end{equation}
where $\tau = \sqrt{-2\log t}$, and when $1/2 \leq t < 1$
\begin{equation}
	\label{eq:approx_phi_2}
	\Phi^{-1}(t) = \tau - \frac{c_0 + c_1 \tau + c_2 \tau^2}{1 + d_1 \tau + d_2 \tau^2 + d_3 \tau^3} + \bar{\epsilon},
\end{equation}
where $\tau = \sqrt{-2\log (1-t)}$. In~\eqref{eq:approx_phi_21} and~\eqref{eq:approx_phi_2},  $c_0 = 2.515517, c_1 = 0.802853, c_2 = 0.010328$, $d_1 = 1.432788, d_2 = 0.189269, d_3 = 0.001308$, and the error term $\abs{\bar{\epsilon}} < 4.5\times 10^{-4}$. Notice that in the range $\tau \in (\sqrt{\log 4}, +\infty)$, the following inequality will be used in the later derivations:
\begin{equation}
	\label{eq:inequality_tau}
	\abs*{\tau - \frac{c_0 + c_1 \tau + c_2 \tau^2}{1 + d_1 \tau + d_2 \tau^2 + d_3 \tau^3}} \leq \tau.
\end{equation}
Figure.~\ref{fig:tau_phi_t} graphically illustrates the behavior of $\abs{\Phi^{-1}(t)}$ and $\tau(t)$ for $10^{-7} < t < 1/2$ and the behavior of $\tau$ and $\tau - \frac{c_0 + c_1 \tau + c_2 \tau^2}{1 + d_1 \tau + d_2 \tau^2 + d_3 \tau^3}$ for $\tau > \sqrt{\log 4}$. $\abs{\Phi^{-1}(t)}$ is bounded by $\tau(t)$ in the range $(10^{-7}, 1/2)$ as plotted in the left panel. The right panel shows that $\tau - \frac{c_0 + c_1 \tau + c_2 \tau^2}{1 + d_1 \tau + d_2 \tau^2 + d_3 \tau^3}$ is positive and bounded by $\tau$ for $\tau > \sqrt{\log 4}$.
\begin{figure}
	\centering
	\includegraphics[width=0.40\linewidth]{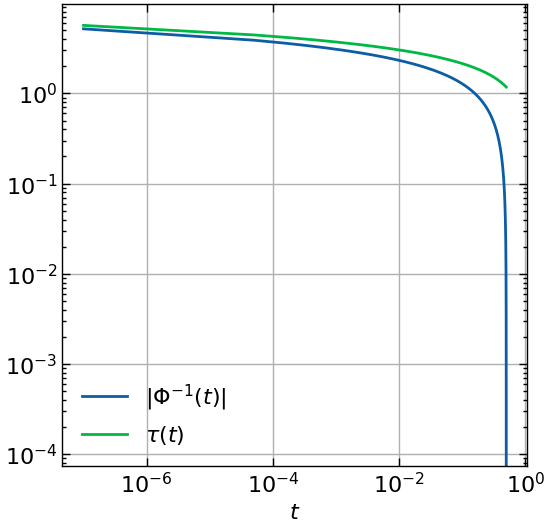}
	\includegraphics[width=0.40\linewidth]{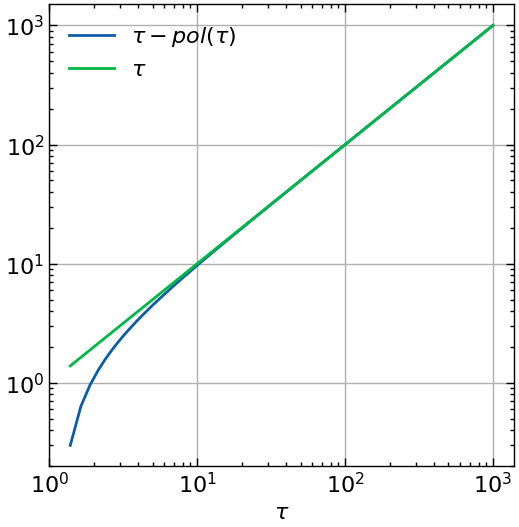}
	\hfill
	\caption{Left: The plot of $\abs{\Phi^{-1}(t)}$ and $\tau(t) = \sqrt{-2\log t}$ for $10^{-7} < t < 1/2$. $\Phi^{-1}$ is computed using the \texttt{scipy.norm} module. Right: Illustrating the inequality~\eqref{eq:inequality_tau} with the plot of $\tau$ and $\tau-pol(\tau) = \tau -  \frac{c_0 + c_1 \tau + c_2 \tau^2}{1 + d_1 \tau + d_2 \tau^2 + d_3 \tau^3}$ for $\tau > \sqrt{\log 4}$. }
	\label{fig:tau_phi_t}
\end{figure}

 For the example considered in this section, we need to examine the behavior of $g$ and its mixed first-order derivatives when $1/2\leq t_j < 1$, $j \in 1:s$ to model the singularity. From the implicit function theorem, the derivative $\frac{\partial \Phi^{-1}(t)}{\partial t}$ is given by,
\begin{equation}
	\label{eq:derivative_phi_-1}
	\frac{\partial \Phi^{-1}(t)}{\partial t} = \sqrt{2\pi} \exp\left(\frac{1}{2}\left(\Phi^{-1}(t)\right)^2\right).
\end{equation}
Substitude the approximation~\eqref{eq:approx_phi_2} into~\eqref{eq:derivative_phi_-1}, we have
\begin{equation}
	\label{eq:derivative_phi_-1_upper_bound}
	\begin{split}
	\frac{\partial \Phi^{-1}(t)}{\partial t} &= \sqrt{2\pi} \exp\left(\frac{1}{2}\left( \tau - \frac{c_0 + c_1 \tau + c_2 \tau^2}{1 + d_1 \tau + d_2 \tau^2 + d_3 \tau^3} + \bar{\epsilon} \right)^2\right)\\
	&= \sqrt{2\pi} \exp\left(\frac{1}{2}\left( \left( \tau - \frac{c_0 + c_1 \tau + c_2 \tau^2}{1 + d_1 \tau + d_2 \tau^2 + d_3 \tau^3} \right)^2 + \bar{\epsilon}^2 + \right. \right.\\
	&\left. \left. 2 \bar{\epsilon} \left( \tau - \frac{c_0 + c_1 \tau + c_2 \tau^2}{1 + d_1 \tau + d_2 \tau^2 + d_3 \tau^3} \right) \right) \right)\\
	&\leq \sqrt{2\pi} \exp\left( \frac{\bar{\epsilon}^2}{2} \right) \exp\left( \frac{1}{2} \tau^2 + \abs{\bar{\epsilon}} \tau \right)\\
	&= \sqrt{2\pi} \exp\left( \frac{\bar{\epsilon}^2}{2} \right) \exp \left( -\log (1-t) + \frac{\abs{\bar{\epsilon}}}{2} \sqrt{-2 \log (1-t)} \right)\\
	&\leq \sqrt{2\pi} \exp\left( \frac{\bar{\epsilon}^2}{2} \right) \exp\left( -(1 + \abs{\bar{\epsilon}} ) \log (1-t) \right)\\
	&= \sqrt{2\pi} \exp\left( \frac{\bar{\epsilon}^2}{2} \right) (1-t)^{-(1 + \abs{\bar{\epsilon}})},
\end{split}
\end{equation}
where in the first inequality we apply the inequality~\eqref{eq:inequality_tau} and in the second inequality we use the fact that $\frac{1}{2} \sqrt{-2 \log (1-t)} \leq -\log(1-t)$ for $t \geq \frac{1}{2}$.
We also have
\begin{equation}
	\begin{split}
	\exp\left(\sum_{j=1}^s \sigma_j \Phi^{-1}(t_j)\right) &= \exp\left(\sum_{j=1}^s \sigma_j (\tau - \frac{c_0 + c_1 \tau + c_2 \tau^2}{1 + d_1 \tau + d_2 \tau^2 + d_3 \tau^3} + \bar{\epsilon}) \right)\\
	&\leq \exp\left(\sum_{j=1}^s \sigma_j (\tau  + \abs{\bar{\epsilon}}) \right)\\
	&= \exp\left( \abs{\bar{\epsilon}} \sum_{j=1}^s \sigma_j \right) \exp \left( \sum_{j=1}^s \sigma_j \sqrt{-2\log (1-t_j)} \right).
	\end{split}
\end{equation}
Using the above derivations, we can find an upper bound for the partial derivative of $g$ as
\begin{align}
	\label{eq:upper_bound_partial_derivative_g_ex1}
	\frac{\partial g}{\partial \mathbf{t}_\mathfrak{u}} &= \exp\left(\sum_{j = 1}^s \sigma_j \Phi^{-1}(t_j) \right) \cdot \prod_{j \in \mathfrak{u}} \sigma_j {\partial^j \Phi^{-1}({t}_j)} \nonumber\\
	&\leq   \exp \left( \sum_{j=1}^s \sigma_j \sqrt{-2\log (1-t_j)} \right) \exp\left( \abs{\bar{\epsilon}} \sum_{j=1}^s \sigma_j \right) \cdot \sqrt{2\pi} \exp\left( \frac{\bar{\epsilon}^2}{2} \right) \prod_{j \in \mathfrak{u}} \sigma_j (1-t_j)^{-(1 + \abs{\bar{\epsilon}})}. 
\end{align}
We mention that the approximations of $\Phi^{-1}$ and its derivative are also applied in~\cite{bartuska2023double} for the asymptotic analysis of a lognormal model. {To check the boundary growth condition~\eqref{eq:boundary_growth_condition}, let $\mathbf{v} = (v_1, \dotsc, v_s)$, we define: } 
\begin{align}
	{h(\mathbf{v}) := \exp\left(\sum_{j = 1}^{s} \sigma_j {\sqrt{-2\log  {v}_j} }\right)} = \exp\left(\sum_{j = 1}^{s} \sigma_j {\sqrt{-2\log (1 - {t}_j)} }\right)  
\end{align}
for the notation simplicity and study the local growth behavior of $h(\mathbf{v})$. 


%
%

\begin{lemma}[First-order Taylor approximation]
	For $\mathbf{v} \in (0, \frac{1}{2})^s$, the function $h(\mathbf{v})$ satisfies the following bound:
	\begin{align}
	{h(\mathbf{v})} \leq B(\mathbf{v}_{c}) \prod_{j=1}^s ({v}_j)^{-A_j(\mathbf{v}_{c})}, \quad \mathbf{v}_{c} \in \partial K_{n,s},
	\label{eq:inequality_local_rate_A}
	\end{align}
	where $B(\mathbf{v}_{c}) = \prod_{j=1}^s \exp\left( \sigma_j \sqrt{\frac{-\log (\mathbf{v}_{c}^j)}{2} } \right)$, $A_j(\mathbf{v}_{c}) = \frac{\sigma_j}{\sqrt{-2 \log \mathbf{v}_c^j}}$. 
	\label{lemma:first_order_taylor_approximation_local_rate_A}
\end{lemma}

The proof is provided in the supplementary material~\ref{sec:proof_of_lemma_first_order_taylor_approximation_local_rate_A}. 
Let us now consider the optimization problem~\eqref{eq:optimization_aj}:
\begin{align}
	\begin{split}
		\max_{\mathbf{v}_{c} \in \partial K_{n,s}}&~ \log  B(\mathbf{v}_c) + \sum_{j=1}^s -A_j \log \mathbf{v}_c^j\\
		s.t.&~ -\sum_{j=1}^s \log \mathbf{v}_c^j = \log \delta^{-1},
	\end{split}
	\label{eq:optimization_aj_ex_1}
\end{align}
for $\delta = Cn^{-1}$. Following the standard Lagrangian approach (the details are refered to supplementary material~\ref{sec:lagrangian_to_optimization_problem_3.7}), the value of $A_j^*$ from~\eqref{eq:optimization_aj}, is given by
\begin{align}
	A_j^* = \sqrt{\frac{\sum_{j=1}^s \sigma_j^2}{2 \log \delta^{-1}}} = \sqrt{\frac{\sum_{j=1}^s \sigma_j^2}{2 (\log n - \log C) } }.
	\label{eq:Astar_exp}
\end{align}
and
\begin{equation}
	\label{eq:Bstar_exp}
	B(\mathbf{v}_c^{*}) = \exp\left( \sqrt{\frac{\sum_{j=1}^s \sigma_j^2}{2} \left( \log n - \log C \right) } \right),
\end{equation}
where $\mathbf{v}_c^{*}$ is the optimizer to the optimization problem~\eqref{eq:optimization_aj_ex_1}. 
Thus we have,
\begin{equation}
	\abs*{\frac{\partial g}{\partial \mathbf{t}_{1:s}}} \leq B_{n, s}^{*} \prod_{j=1}^s (1 - t_j)^{-(1 + \abs{\bar{\epsilon}}) + \max_j A_j^{*} }, 
\end{equation}
where $B_{n, s}^{*} = B(\mathbf{v}_c^{*}) K_1 $, $K_1 = \exp\left( \abs{\bar{\epsilon}} \sum_{j=1}^s \sigma_j \right) \sqrt{2\pi} \exp\left( \frac{\bar{\epsilon}^2}{2} \right) \prod_{j =1}^s \sigma_j$. With an abuse of notation, now we perturb the values of $A_j^{*}$, $j = 1, \dotsc, s$ to have
\begin{equation}
	A_j^{*} = (j-1) \bar{\delta} +   \sqrt{\frac{\sum_{j=1}^s \sigma_j^2}{2 (\log n - \log C) } },
\end{equation}
for $\bar{\delta}  > 0$ and $A_j^{*} < 1$. The expected integration error for the RQMC method satisfies
	\begin{align}
		\mathbb{E} [ \lvert I(g) - \hat{I}_n^r(g) \rvert ] \leq C_1(n) n^{-1 + \max_j A_j^*} + C_2(n) n^{-1 + \epsilon + \max_j A_j^*}
		\label{eq:QMC_convergence_model_ex1}
	\end{align}
where $\epsilon > 0$, and $C_1(n) = 2^{s+1} \tilde{B}_{n,s} \sum_{j=1}^s \frac{C^{1-A_j^{*}}}{(A_j^{*} - 1) \phi^{\prime}(A_j^{*}) } $ and \\ $C_2(n) = C_{\epsilon, s} B^{*}_{n, s} \left( \frac{5}{2} \right)^s \sum_{\substack{\mathfrak{u} \subseteq 1:s \\ \mathfrak{u} \neq \varnothing}} \frac{1}{A_{m(\mathfrak{u})}^{*}} \prod_{j \in \mathfrak{u}_{-}} \frac{1}{A_{m(\mathfrak{u})}^{*} -A_j^*} (C^{-1})^{\max_{j} A_{j}^{*} }$ follows the same definition as in Theorem~\ref{thm:nonasymptotic_QMC_error_estimate}. More specficially, we can find the upper bounds for $\tilde{B}_{n, s}$, $C_1(n)$ and $C_2(n)$ as
\begin{align}
	\tilde{B}_{n, s} &\leq 2^s {B}_{n, s}^{*} \prod_{j=1}^s (1 + \frac{1}{A_j^{*}}) \nonumber \\ 
	& \leq 2^{2s} B(\mathbf{v}_c^{*}) K_1 \sqrt{\frac{2 (\log n - \log C) }{\sum_{j=1}^s \sigma_j^2} } \nonumber \\ 
	&\leq 2^{2s} K_1 \exp\left( \sqrt{\frac{\sum_{j=1}^s \sigma_j^2}{2} \left( \log n - \log C \right) } \right)  \sqrt{\frac{2 (\log n - \log C) }{\sum_{j=1}^s \sigma_j^2} } \\
	C_1(n) &\leq 2^{s+1} \tilde{B}_{n, s} \  \bar{\delta}^s \frac{\sum_{j=1}^s C^{1 - A_j^{*}}}{\max_j A_j^{*} - 1} \nonumber  \\
	&\leq K_2 \exp\left( \sqrt{\frac{\sum_{j=1}^s \sigma_j^2}{2} \left( \log n - \log C \right) } \right)  \sqrt{\frac{2 (\log n - \log C) }{\sum_{j=1}^s \sigma_j^2} } \frac{\sum_{j=1}^s C^{1 - A_j^{*}}}{\max_j A_j^{*} - 1}\\
	C_2(n) &\leq C_{\epsilon, s} B^{*}_{n, s} \left( \frac{5}{2} \right)^s \sum_{\substack{\mathfrak{u} \subseteq 1:s \\ \mathfrak{u} \neq \varnothing}} \frac{1}{A_{m(\mathfrak{u})}^{*}} \prod_{j \in \mathfrak{u}_{-}} \frac{1}{A_{m(\mathfrak{u})}^{*} -A_j^*} (C^{-1})^{\max_{j} A_{j}^{*} } \nonumber\\
	&\leq C_{\epsilon, s} B^{*}_{n, s} \left( \frac{5}{2} \right)^s \sqrt{\frac{2 (\log n - \log C) }{\sum_{j=1}^s \sigma_j^2} } (C^{-1})^{\max_{j} A_{j}^{*} } \sum_{\substack{\mathfrak{u} \subseteq 1:s \\ \mathfrak{u} \neq \varnothing}} \bar{\delta}^{\abs{\mathfrak{u}_{-}}} \nonumber \\
	&\leq C_{\epsilon, s} K_1 \exp\left( \sqrt{\frac{\sum_{j=1}^s \sigma_j^2}{2} \left( \log n - \log C \right) } \right) \left( \frac{5}{2} \right)^s \sqrt{\frac{2 (\log n - \log C) }{\sum_{j=1}^s \sigma_j^2} } \nonumber \\
	& \hspace*{10em} \cdot (C^{-1})^{\max_{j} A_{j}^{*} } \sum_{k=1}^s {{s}\choose{k}} \bar{\delta}^{k-1} \nonumber \\
	&\leq K_3  \exp\left( \sqrt{\frac{\sum_{j=1}^s \sigma_j^2}{2} \left( \log n - \log C \right) } \right)  \sqrt{\frac{2 (\log n - \log C) }{\sum_{j=1}^s \sigma_j^2} },
\end{align}
where $K_2 = 2^{3s+1} K_1 \bar{\delta}^s, K_3 = C_{\epsilon, s} K_1 \left( \frac{5}{2} \right)^s (C^{-1})^{\max_{j} A_{j}^{*} }  s! \frac{1 - \bar{\delta}^{k}}{1-\bar{\delta}}$.
Similar to Example~\ref{example:2}, one can show that $C_1(n), C_2(n) = o(n^{\alpha})$ for $\alpha > 0$. For a finite sample size $n$, considering the error estimate~\eqref{eq:QMC_convergence_model} in Theorem~\ref{thm:nonasymptotic_QMC_error_estimate}, the exponent on $n$ decreases when $n$ increases, or the variance $\sigma_j^2$ decreases for each dimension $j$. 


\subsection{Elliptic partial differential equations with lognormal coefficients}
\label{subsec:qmc4pde}
We consider an elliptic PDE on a Lipschitz domain $\mathcal{D} \subset \mathbb{R}^d$, in the following form:
	\begin{subequations}
	\label{eq:bvp_general}
	\begin{align}
	\label{eq:pde_general}
	-\nabla_{\mathbf x} \cdot \left( a(\mathbf x; \omega) \nabla_{\mathbf x} u(\mathbf x; \omega) \right) 
	&= f(\mathbf x; \omega) &&\text{for $\mathbf x  \in \mathcal D$,} 
	\\
	u(\mathbf x; \omega) &=  0 &&\text{for $\mathbf x \in \partial\mathcal{D}$},
	\end{align}
\end{subequations}
with almost all $\omega \in \Omega$, where $\Omega$ belongs to a complete probability space tuple $(\Omega, \mathcal{F}, \mathbb{P})$. The differential operators $\nabla \cdot$ and $\nabla$ are taken w.r.t. the spatial variable $\mathbf x$. We {explore the} equations in~\eqref{eq:bvp_general} in the weak form. For a suitable functional space $V$ (e.g., $V = H^1_0(\mathcal{D})$), we seek $u(\mathbf{x}; \omega) \in V$ such that
\begin{equation}
	\begin{split}
	\left \langle a(\mathbf x; \omega) \nabla_{\mathbf x} u(\mathbf x; \omega), \nabla_{\mathbf x} v(\mathbf x; \omega) \right\rangle & = \left \langle f(\mathbf x; \omega), v(\mathbf x; \omega) \right\rangle \quad\\
	&  \text{for all } v \in V \text{ and almost all } \omega \in \Omega,
	\end{split}
\end{equation}
where $\left\langle \cdot, \cdot \right \rangle$ denotes the inner product. The QoI $Q(\omega)$ is given in the following general form:
\begin{align}
	Q(\mathbf{y}(\omega)) = \mathcal{G}(u(\cdot ;\omega)),
\end{align}
where $\mathcal{G} \in V^\prime$, $V^\prime$ denotes the dual space of $V$. We are interested in computing $\mathbb{E}[Q]$. The integrand in~\eqref{eq:integration_problem} is given by $g = Q \circ {\Phi}^{-1}$. 

We consider $a(\mathbf{x};\omega)$, which takes the following form:
\begin{align}
a(\mathbf{x};\omega) = \exp\left(\sum_{j = 1}^s {y}_j { (\omega)} \psi_j { (\mathbf{x})}\right),
\end{align}
with $s \in \mathbb{N}^+ $ and $y_j$ are i.i.d. samples from $\mathcal{N}(0,1)$, $j = 1,\dotsc, s$. 
Moreover, we define $b_j = \lVert \psi_j \rVert_{L^{\infty}(\mathcal{D})} < +\infty$, for $j = 1, \dotsc, s$.


Similar to the analysis in Section~\ref{subsec:lognormal}, we can determine $A_j^*$~\eqref{eq:a_star} as follows:
\begin{align}
A_j^* = \sqrt{\frac{\sum_{j=1}^s b_j^2}{2 (\log n - \log C) } },
\label{eq:Astar_pde}
\end{align}
and $\max_j A_j^* = \sqrt{\frac{\sum_{j=1}^s b_j^2}{2  (\log n - \log C) }}$. The details are provided in the supplementary material~\ref{sec:derivation_of_astar_pde}. The rest of the RQMC error analysis for this example is identical to the one in Section~\ref{subsec:lognormal}, except that the coefficients $b_j$ are used instead of $\sigma_j$. 

{In this section, we have considered two concrete examples with unbounded variation and applied the theory developed in Section~\ref{sec:non_asymptotic_convergence_rate_QMC} to analyze their convergence rates.} Consistent with the theory, the convergence rate can be improved by increasing the sample size $n$. However, whether we can further enhance the convergence rate by modifying the sample distribution and placing more points near the singular corners is unclear. In the next section, we discuss the effects of importance sampling.

\section{Importance sampling}
\label{sec:importance_sampling}

This section provides two kinds of IS proposal distributions and analyzes their effects on the convergence rate. We will study the asymptotic behavior of the QMC with IS integrand and demonstrate the improved asymptotic convergence rate.

\subsection{First  proposal distribution}
\label{subsec:first_integral}

This section proposes a Gaussian distribution with scaled variance to distribute more samples closer to the corners. This kind of IS was studied in~\cite{Kritzer2020efficient} for multivariate functions that belong to certain Sobolev spaces, where the worst-case integration error was analyzed and optimized. 

%
Let $g \in \{\exp \circ (\boldsymbol{\sigma}^T \Phi^{-1}), Q \circ \Phi^{-1} \} : [0, 1]^s \to \mathbb{R}$, which are the two integrands considered in Section~\ref{sec:integrand_infty}. We introduce the component-wise multiplication notation $\odot$, such that $\bm{\alpha} \odot \mathbf{y} = \{\alpha_1\mathrm{y}_1, \alpha_2\mathrm{y}_2, \dotsc, a_s \mathrm{y}_s\} = \mathrm{diag}(\bm{\alpha}) \mathbf{y}$ when $\boldsymbol{\alpha}, \mathbf{y} \in \mathbb{R}^s$. The integral of $g$ over $[0, 1]^s$ is
\begin{equation}
\begin{split}
I(g) &= \int_{[0,1]^s} g(\mathbf{t}) d\mathbf{t}\\
&= \int_{\mathbb{R}^s} \nu (\mathbf{y}) \rho(\mathbf{y}) d\mathbf{y}\\
&= \prod_{j=1}^s {\alpha_j} \int_{\mathbb{R}^s}  \nu (\bm{\alpha} \odot \mathbf{y}) \rho({\bm{\alpha}} \odot \mathbf{y}) d\mathbf{y}\\
&= \prod_{j=1}^s {\alpha_j} \int_{[0, 1]^s} \nu (\bm{\alpha}\odot \Phi^{-1}(\mathbf{t})) \cdot \frac{\rho(\bm{\alpha} \odot \Phi^{-1}(\mathbf{t}))}{\rho(\Phi^{-1}(\mathbf{t}))}d\mathbf{t},
\end{split}	
\end{equation}
where $\nu = g \circ \Phi: \mathbb{R}^s \to \mathbb{R}$, $\bm{\alpha} \in \mathbb{R}^s_{+}$, and $\rho$ is the $s$-dimensional standard normal distribution density. The integrand with IS $g_{\textrm{IS}}$ is given by
\begin{align}
\begin{split}
g_{\textrm{IS}}(\mathbf{t}) &= \left(\prod_{j=1}^s {\alpha_j} \right) \nu (\bm{\alpha} \odot \Phi^{-1}(\mathbf{t})) \cdot \frac{\rho(\bm{\alpha} \odot \Phi^{-1}(\mathbf{t}))}{\rho(\Phi^{-1}(\mathbf{t}))}\\
&= \left(\prod_{j=1}^s {\alpha_j} \right) \nu (\bar{\mathbf{y}}) \cdot  \prod_{j=1}^s  \exp\left( -\frac{\mathrm{y}_j^2}{2} ({\alpha}_j^2 - 1) \right),
\end{split}
\end{align}
where we define $\mathbf{y} = \Phi^{-1}(\boldsymbol{t})$ and $\bar{\mathbf{y}} = \bm{\alpha} \odot \mathbf{y} = \bm{\alpha} \odot \Phi^{-1}(\boldsymbol{t})$. The function $g_{\textrm{IS}}$ no longer has the singularity at the boundaries for ${\alpha}_j > 1$, $j = 1, 2, \dotsc, s$. The detailed study of the derivative ${\partial^{\mathfrak{u}}}g_{\mathrm{IS}}$ is provided in the supplementary material~\ref{sec:derivation_of_partial_g_is}.
%
We now switch to the analysis of the derivative in Example 2, as we will find later that the conclusions will apply to Example 1 as well. 


In Example 2, $\nu = Q$, we apply the analysis in~\ref{sec:derivation_of_partial_g_is_2} and obtained the approximate upper bound of the derivative ${\partial^{\mathfrak{u}}}g_{\mathrm{IS}}$:
\begin{align}
	\mathcal{O} \left(\prod_{j=1}^s t_j^{-\mathbbm{1}_{j \in \mathfrak{u}} + \alpha_j^2 - 1 - \delta }\right)  \quad  t_j \to 0, \ j = 1, \dotsc, s, \delta > 0.
	\label{eq:partial_g_A}
\end{align}
From~\cite{Owen06}, we obtain the RQMC convergence rate $\mathcal{O}(n^{-1+\epsilon - \min_{j=1}^s \alpha_j^2 + 1 + \delta})$. Choosing $\alpha_j > 1$ will lead to an improved convergence rate. In the following we provide one approach to find one $\bm{\alpha}^*$ that minimizes the variance of the integrand $g_{\mathrm{IS}}$. 

We aim to determine $\mathbf{\bm{\alpha}}^*$ by minimizing the variance of $g_{\mathrm{IS}}$:
\begin{align}
\textrm{Var}(g_{\mathrm{IS}}) = \mathbb{E}[g_{\mathrm{IS}}^2] - (\mathbb{E}[g_{\mathrm{IS}}])^2,
\end{align}
which is equivalent to finding the minimizer for the second-order moment of $g_{\mathrm{IS}}$, as $\mathbb{E}[g_{\mathrm{IS}}] = \mathbb{E}[g]$: 
\begin{align}
\begin{split}
\bm{\alpha}^* &= \argmin_{\bm{\alpha} > \bm{1}} \mathbb{E}[ g_{\mathrm{IS}} ^2(\mathbf{t})]\\
&= \argmin_{\bm{\alpha} > \bm{1}} \int_{[0, 1]^s} \left(\prod_{j=1}^s {\alpha_j} \right)^2 \nu^2(\bm{\alpha} \odot \Phi^{-1}(\mathbf{t})) \cdot \frac{\rho^2(\bm{\alpha} \odot \Phi^{-1}(\mathbf{t}))}{\rho^2(\Phi^{-1}(\mathbf{t}))} d \mathbf{t}\\
&= \argmin_{\bm{\alpha} > \bm{1}} {\int_{\mathbb{R}^s} \left(\prod_{j=1}^s {\alpha_j} \right)^2 \nu^2(\bm{\alpha} \odot \mathbf{y}) \cdot \frac{\rho^2(\bm{\alpha} \odot \mathbf{y})}{\rho^2(\mathbf{y})} \rho(\mathbf{y}) d \mathbf{y}}\\
&= \argmin_{\bm{\alpha} > \bm{1}} {\int_{\mathbb{R}^s} \left(\prod_{j=1}^s {\alpha_j} \right) \nu^2( \mathbf{y}) \cdot \frac{\rho^2( \mathbf{y})}{\rho^2(\bm{\frac{1}{\alpha}} \odot \mathbf{y})} \rho(\bm{\frac{1}{\alpha}} \odot \mathbf{y}) d \mathbf{y}}\\
&= \argmin_{\bm{\alpha} > \bm{1}} {\int_{\mathbb{R}^s} \left(\prod_{j=1}^s {\alpha_j} \right) \nu^2( \mathbf{y}) \cdot \frac{\rho( \mathbf{y})}{\rho(\bm{\frac{1}{\alpha}} \odot \mathbf{y})} \rho( \mathbf{y}) d \mathbf{y} }.
\end{split}
\end{align}
When the analytic solution is unavailable, we aim to find an approximate optimizer $\bar{\bm{\alpha}}$ using $n$ QMC samples. Specifically, we minimize the following objective function
\begin{align}
\bar{\bm{\alpha}} = \argmin_{\bm{\alpha} > \bm{1}}  \left(\prod_{j=1}^s {\alpha_j} \right) \frac{1}{n} \sum_{i=1}^n Q^2( \mathbf{y}_i) \cdot \frac{\rho( \mathbf{y}_i)}{\rho(\bm{\frac{1}{\alpha}} \odot \mathbf{y}_i)}.
\end{align}
where $\mathbf{y}_i = \Phi^{-1}(\mathbf{t}_i)$, with $\mathbf{t}_i$ the i-th QMC point. In practice, the number $n$ in the pilot run does not exceed the number of simulations samples. 

\subsection{Second proposal distribution}
\label{subsec:second_integral}
This section considers an IS proposal distribution on $[0, 1]^s$. We will only analyze Example 2 for brevity, as the case for Example 1 can be derived similarly. Inspired by the beta distribution, we propose the following distribution $\rho_{\beta}: [0, 1] \to \mathbb{R}_{+}$:
\begin{align}
\rho_{\beta}({t} ) = 
\begin{cases}
C t^{\beta - 1} & \text{$0 \leq t < \frac{1}{2}$}\\
C (1 - t)^{\beta - 1} & \text{$\frac{1}{2} \leq t \leq 1$},
\end{cases} 
\end{align}
with the constant $C = \beta \cdot (\frac{1}{2})^{1-\beta}$. The reason to propose such a distribution is the ease of computing the CDF $\Phi_{{\beta}} : [0, 1] \to [0, 1]$ and its inverse $\Phi_{{\beta}}^{-1}: [0, 1]\to [0, 1]$:
\begin{align}
\Phi_{{\beta}}({t} ) = 
\begin{cases}
\frac{C}{\beta} t^{\beta} & \text{$0 \leq t < \frac{1}{2}$}\\
1 - \Phi_{{\beta}}(1 - {t} ) & \text{$\frac{1}{2} \leq t \leq 1$},
\end{cases} 
\end{align}
\begin{align}
\Phi^{-1}_{\beta} (w) = (2w)^{\frac{1}{\beta}} \cdot \frac{1}{2} \quad \textrm{where } 0\leq w < 1/2.
\end{align}
The distribution extends to multidimension by tensor products of the one-dimensional. Next, we apply IS to compute the integral $I(g)$. Using the density $\rho_{\bm{\beta}}$, we can write
\begin{align}
\begin{split}
I(g) &=\int_{[0, 1]^s} g( \mathbf{t} ) d\mathbf{t}\\
&=\int_{[0, 1]^s} g( \mathbf{t}) \cdot \frac{\rho_{\bm{\beta}}(\mathbf{t})}{\rho_{\bm{\beta}}(\mathbf{t})} d\mathbf{t}\\
&=\int_{[0, 1]^s} g( \Phi_{\bm{\beta}}^{-1} (\mathbf{w})) \cdot \frac{1}{\rho_{\bm{\beta}}(\Phi_{\bm{\beta}}^{-1} (\mathbf{w}))} d\mathbf{w}.
\end{split}
\end{align}
In Example 2, $g = Q\circ \Phi^{-1}$, and we introduce
\begin{align}
g_{\mathrm{IS}} (\mathbf{w}) = Q\left(\Phi^{-1} (\Phi_{\bm{\beta}}^{-1} (\mathbf{w}))\right) \cdot \frac{1}{\rho_{\bm{\beta}}(\Phi_{\bm{\beta}}^{-1} (\mathbf{w}))}.
\end{align}
We obtained the approximate upper bound of the derivative ${\partial^{\mathfrak{u}}}g_{\mathrm{IS}}$:
\begin{equation}
\mathcal{O} \left( \prod_{j=1}^s w_j^{-\mathbbm{1}_{j \in \mathfrak{u}} + \frac{1}{\beta_j} - 1 - \delta} \right) \quad  w_j \to 0, \ j = 1, \dotsc, s, \delta > 0.
\end{equation}
From~\cite{Owen06}, we obtain the RQMC convergence rate $\mathcal{O}(n^{-1+\epsilon - \min_{j=1}^s \frac{1}{\beta_j} + 1 + \delta})$. Choosing $\beta_j < 1$ will lead to an improved convergence rate.
The details of the derivations of $\partial^{\mathfrak{u}} g_{\mathrm{IS}}$ and are provided in the supplementary material~\ref{sec:derivation_of_partial_g_is_beta}. 
Similar to the last section, a possible choice of the parameter $\bm{\beta}$ can be determined by minimizing the second-order moment of $g_{\mathrm{IS}}$:
\begin{align}
\begin{split}
\beta^* &= \argmin_{\bm{\beta} < 1} \mathbb{E}[ g_{\mathrm{IS}}^2(\mathbf{w})]\\
&= \argmin_{\bm{\beta} < 1} \int_{[0, 1]^s}  \left(\frac{g(\Phi^{-1} (\mathbf{t}))}{\rho_{\bm{\beta}}(\mathbf{w})} \right)^2 \cdot \rho_{\bm{\beta}}(\mathbf{w}) d \mathbf{w}\\
&= \argmin_{\bm{\beta} < 1} \int_{[0, 1]^s} \frac{g^2(\Phi^{-1} (\mathbf{w}))}{\rho_{\bm{\beta}}(\mathbf{w})}  d \mathbf{w}.
\end{split}
\end{align}
Again, we seek an optimizer based on ensembles when the analytical solution is unavailable.

\section{Numerical results}
\label{sec:numerical}
This section tests the convergence rates of the two examples we have considered, compares them with the analysis, and demonstrates how IS improves the rates. 

\subsection{Lognormal random variable expectation}
We approximate the expectation of a lognormal random variable, that is, 
\begin{equation*}
	\E{g} = \E{\exp\left(\sum_{j=1}^s \sigma_j \xi_j \right)},
\end{equation*}
where $\xi_j$ are i.i.d. $\mathcal{N}(0, 1)$. The expectation is $\prod_{j=1}^{s} \exp\left(\frac{\sigma_j^2}{2}\right)$. We test the convergence of the root mean squared error (RMSE) of the RQMC method using the Sobol' sequence. The RMSE for QMC is computed as follows:
\begin{align}
	\label{eq:rmse_qmc}
	\sqrt{\mathbb{E} \left[ \left(\frac{1}{R} \sum_{r=1}^{R} \frac{1}{n} \sum_{i=1}^n g(t_i \oplus \Delta_r) - \mathbb{E}[g]\right)^2\right]} 
	\approx \sqrt{\frac{1}{R(R-1)} \sum_{r=1}^R \left( \hat{I}_n^{(r)} - \hat{I}_{n}^R \right)^2}
\end{align}
where $\hat{I}_n^{(r)} = \frac{1}{n}  \sum_{i=1}^n g(t_i \oplus \Delta_r) $ and $\hat{I}_{n}^R =  \frac{1}{R} \sum_{r=1}^{R} \hat{I}_n^{(r)}$. 

With the nested uniform scrambling (Owen's scrambling) employed as the randomization, the convergence rate for the RMSE is $\mathcal{O}(n^{-3/2+\epsilon})$ if the integrand has Lipschitz-continuous mixed first-order derivatives~\cite{owen1998scrambling}. A recent study~\cite{liu2024randomized} reveals that the Lipschitz condition can be relaxed for an almost $\mathcal{O}(n^{-3/2+\epsilon})$ rate. However, the integrand $g$ has a singularity at 0 and will not have the above mentioned convergence rate. In both examples presented in this work, we opt for the random linear scramble~\cite{matouvsek1998thel2, hong2003algorithm} as the randomization. This randomization incurs lower computational cost compared to Owen's scrambling while still yielding the same variance for the RQMC estimator. (For instance, refer to Definition 6 of~\cite{owen2008local} or Section 6.12 of~\cite{DKS2013}).

Figures~\ref{fig:example_1_s1}, \ref{fig:example_1_s3}, and \ref{fig:example_1_s6} plot the convergence of the RMSE for the original integrand $g$ and the integrand with importance sampling, $g_{\textrm{IS}}$. Each data point in these figures is RQMC estimator with the number of randomizations $R = 30$. Each boxplot consists of 10,000 samples. As can be observed, the IS significantly reduces the RMSE and improves the convergence rate. 

Tables~\ref{tab:convergence_rate_table_s1}, \ref{tab:convergence_rate_table_s3}, and \ref{tab:convergence_rate_table_s6} provide the measured convergence rates $\gamma$ and the associated values $(1-\gamma) / \sqrt{\sum_{j=1}^s \sigma_j^2}$. In these numerical results, we estimate the convergence rates for Example 1 under various settings by conducting a linear regression on data ranging from $n = 2^{20}$ to $n = 2^{22}$ on a log-log scale. The slope of this regression line will serve as our estimate for the convergence rate at $n = 2^{21}$. For the same sample size $n$, the convergence rate $\gamma$ decreases as $\sqrt{\sum_{j=1}^s \sigma_j^2}$ increases. The value $(1-\gamma) / \sqrt{\sum_{j=1}^s \sigma_j^2}$ exhibits similarity in magnitude for all cases in each dimension setting and increases as the dimension increases. 

\begin{table}[h!]
	\centering
	\caption{Convergence analysis of Example 1 for dimension $s=1$ and standard deviation $\sigma = \{1.0, 2.0, 3.0\}$. The convergence rate is fitted at $n = 2^{21}$, using a linear regression based on the data from $n = 2^{20}$ to $n = 2^{22}$. } \label{tab:convergence_rate_table_s1}
	\begin{tabular}{|c|c|c|c|}
	\hline
	 $\sigma$ & Conv. Rate $\gamma$ & Conv. Rate with IS & $(1 - \gamma) / \sigma $ \\ \hline
	1.0	& 0.875 & 1.888 & 0.125  \\ \hline
	2.0	& 0.719 & 1.340 & 0.141  \\ \hline
	3.0	& 0.582 & 1.088 & 0.139  \\ \hline
	\end{tabular}
\end{table}

\begin{figure}[htbp]
	\centering
	\includegraphics[width = 0.48\textwidth]{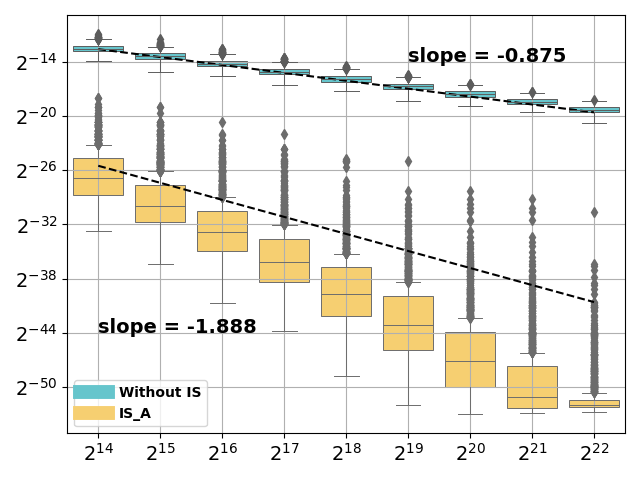}
	\hfill
	\includegraphics[width = 0.48\textwidth]{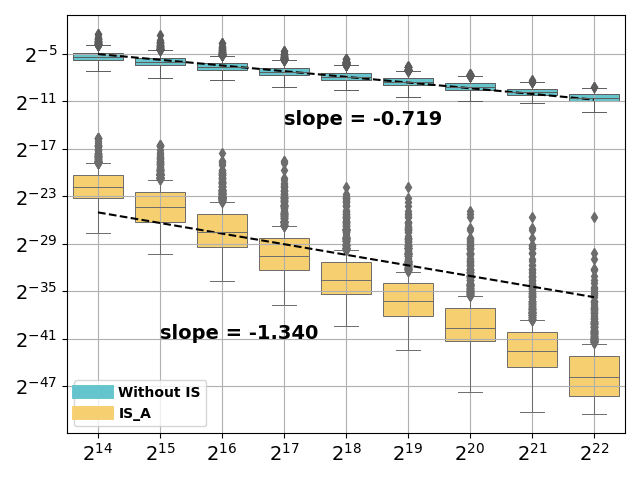}
	\\
	\includegraphics[width = 0.48\textwidth]{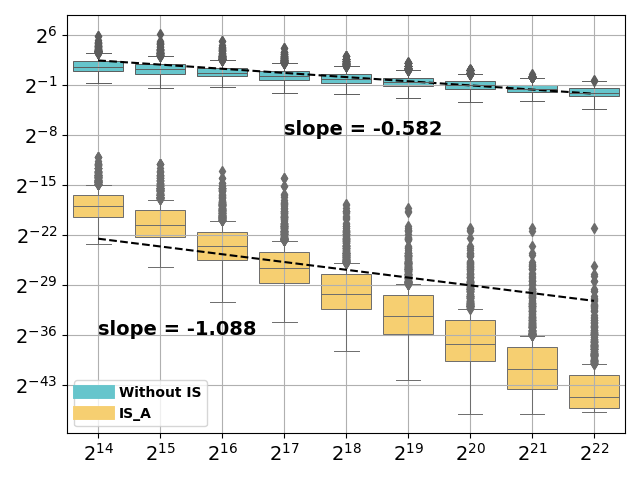}
	\caption{Example 1, the RMSE with $s=1$: $\sigma = 1.0$ (top left), 2.0 (top right), and 3.0 (bottom). Within each box, the grey horizental line marks the median value of a total 10,000 samples. Each box extends from first quartile to the third quartile of each group and each whisker marks the 0 and 99.9 percentile. The grey diamonds denote the outliers. The boxplot settings apply to the following boxplot figures. The convergence rate is fitted at $n = 2^{21}$. }
	\label{fig:example_1_s1}
\end{figure}


\begin{figure}[htbp]
	\centering
	\includegraphics[width = 0.48\textwidth]{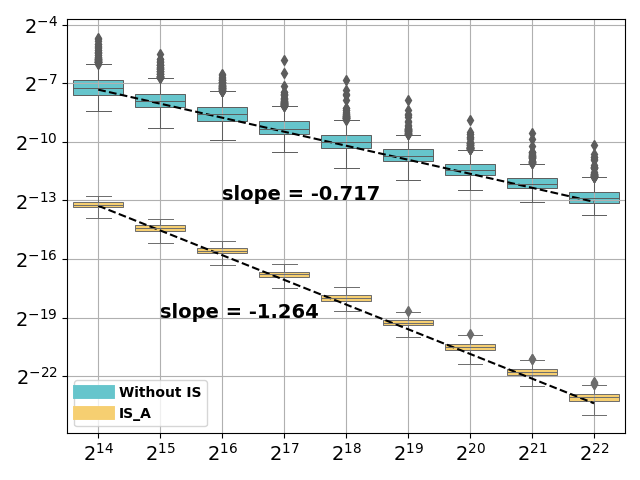}
	\hfill
	\includegraphics[width = 0.48\textwidth]{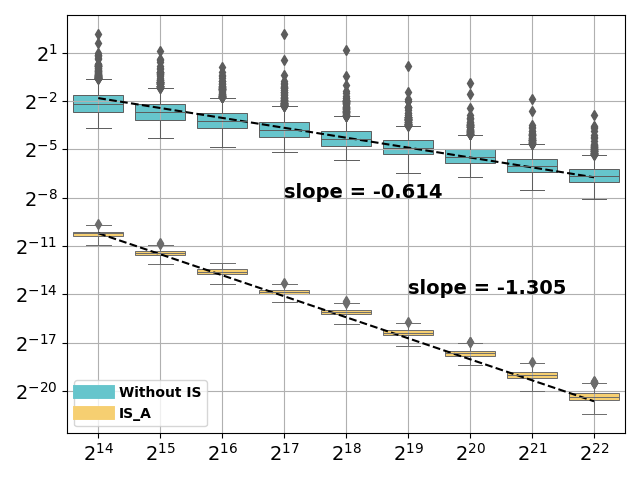}
	\\
	\includegraphics[width = 0.48\textwidth]{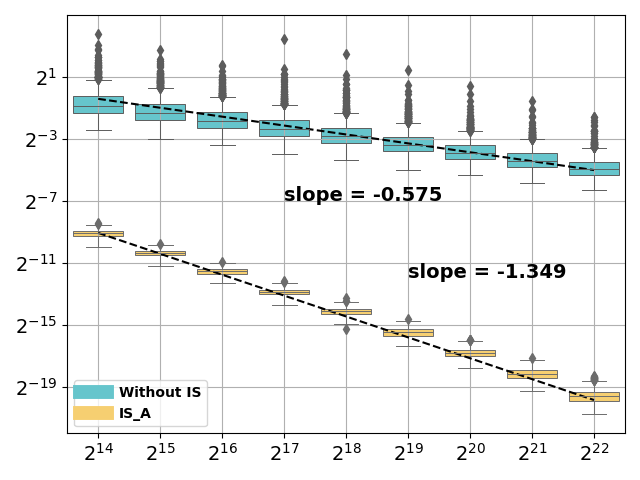}
	\hfill
	\includegraphics[width = 0.48\textwidth]{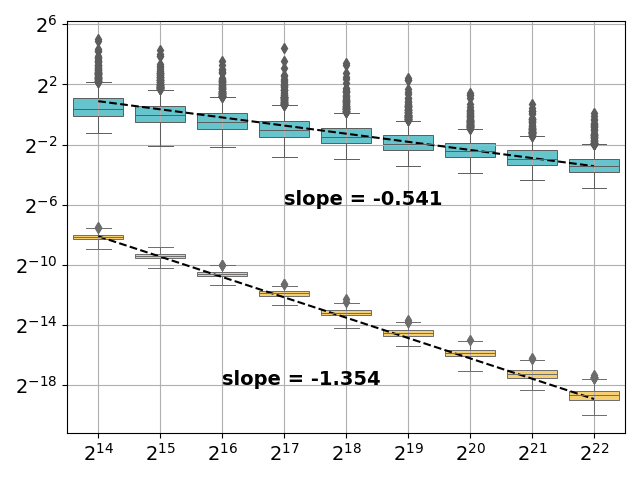}

	\caption{Example 1, the RMSE with $s=3$: $\bm{\sigma} = (1.0, 1.0, 1.0)$ (top left), $(2.0, 1.0, 1.0)$ (top right), $(2.0, 1.4, 1.0)$ (bottom left), and $(1.0, 1.7, 1.0)$ (bottom right). The convergence rate is fitted centered at $n = 2^{21}$.  }

	\label{fig:example_1_s3}
\end{figure}

\begin{figure}[htbp]
	\centering
	\includegraphics[width = 0.48\textwidth]{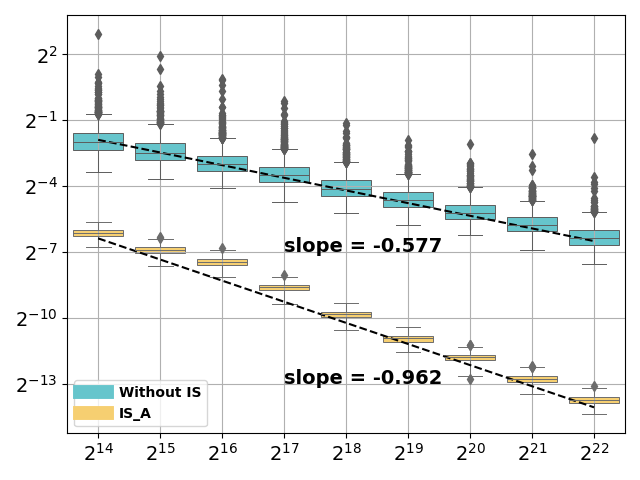}
	\hfill
	\includegraphics[width = 0.48\textwidth]{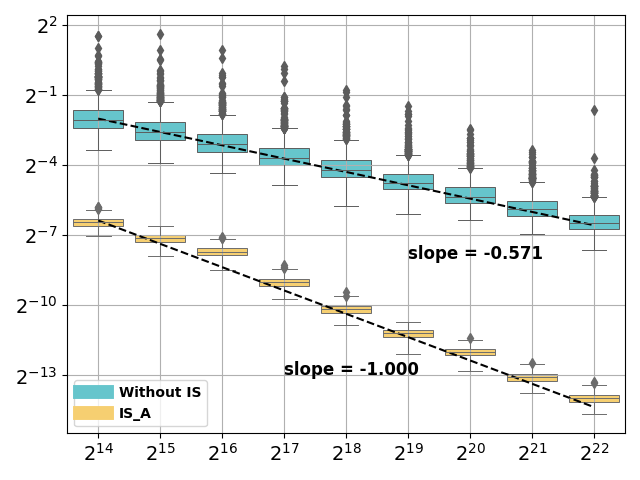}
	\\
	\includegraphics[width = 0.48\textwidth]{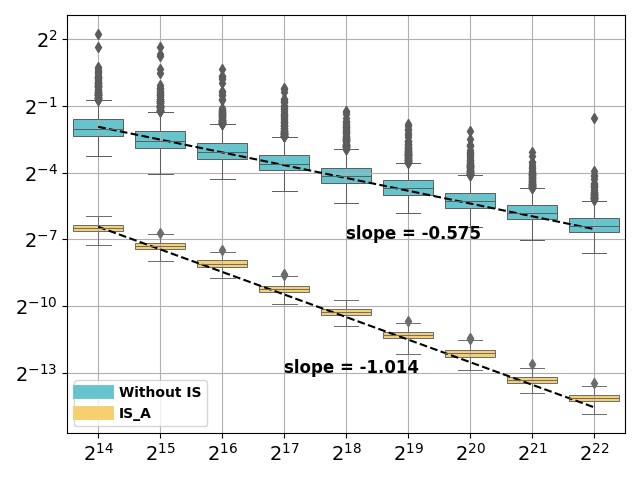}
	\hfill
	\includegraphics[width = 0.48\textwidth]{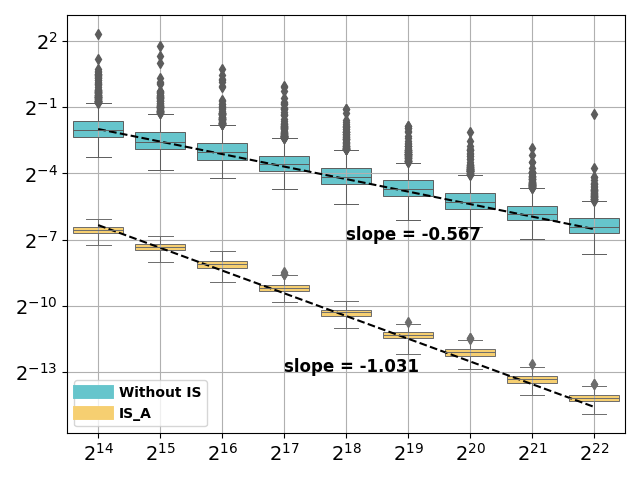}
	\caption{Example 1, the RMSE with $s=6$: Case I (top left), Case II (top right), Case III (bottom left), and Case IV (bottom right). The convergence rate is fitted at $n = 2^{21}$.}
	\label{fig:example_1_s6}
\end{figure}

\begin{table}[h!]
	\centering
	\caption{Convergence of Example 1, for dimension $s=3$ and various standard deviation settings, where the convergence rate is fitted at $n = 2^{21}$.  \label{tab:convergence_rate_table_s3}}
	\begin{tabular}{|c|c|c|c|}
		\hline
		$\bm{\sigma}$ & Conv. Rate $\gamma$ & Conv. Rate with IS & $(1 - \gamma) / \sqrt{\sum_{j=1}^s \sigma_j^2}$ \\ \hline
		(1.0, 1.0, 1.0)	& 0.717 & 1.264 & 0.163 \\ \hline
		(2.0, 1.0, 1.0)	& 0.614 & 1.305 & 0.158 \\ \hline
		(2.0, 1.4, 1.0)	& 0.575 & 1.349 & 0.161 \\ \hline
		(2.0, 1.7, 1.0)	& 0.541 & 1.354 & 0.162 \\ \hline
	\end{tabular}
\end{table}

\begin{table}[h!]
	\centering
	\caption{Convergence of Example 1 for dimension $s=6$ and various standard deviation $\bm{\sigma}$ settings, where ${\sigma}_j = {1}$, $j = 1, \dotsc, s$ for Case I, and $\sigma_j = \frac{\sqrt{6} \xi_j }{\sqrt{\sum_{j=1}^s \xi_j^2}}$, where $\xi_j \sim \textrm{Lognormal}(0, 1)$, $j = 1, \dotsc, s$ for Cases II, III, and IV. The convergence rate is fitted at $n = 2^{21}$.  \label{tab:convergence_rate_table_s6}}
	\begin{tabular}{|c|c|c|c|}
		\hline
		$\bm{\sigma}$ cases & Conv. Rate $\gamma$ & Conv. Rate with IS & $(1 - \gamma) / \sqrt{\sum_{j=1}^s \sigma_j^2}$ \\ \hline
		I	& 0.577 & 0.962 & 0.173 \\ \hline
		II	& 0.571 & 1.000 & 0.175 \\ \hline
		III	& 0.575 & 1.014 & 0.174 \\ \hline
		IV	& 0.567 & 1.031 & 0.178 \\ \hline
	\end{tabular}
\end{table}


\newpage
\subsection{Elliptic partial differential equations with lognormal coefficient}
We first specify the settings of this example. The QoI is the weighted integration of the solution $u$ over the entire domain $\mathcal{D}$, that is,
\begin{align*}
	Q(u) = \int_{\mathcal{D}} g(\mathbf{x}) u(\mathbf{x};\cdot) d\mathbf{x},
\end{align*}
where $\mathcal{D} = [-1, 1]^2$, $g = \rho(;\bm{\mu}, \bm{\Sigma}) * \mathbbm{1}_{D_0}$, $*$ denotes the convolution operator, $\mathbbm{1}$ is the indicator function, $\mathcal{D}_0 = \left[0.25, 0.5\right] \times \left[-0.5, -0.25\right]\subset \mathcal{D}$, and $\rho(;\bm{\mu}, \bm{\Sigma})$ represents the Gaussian density function in $\mathbb{R}^2$ with mean $\bm{\mu} = 0$ and covariance $\bm{\Sigma} = \frac{1}{16} \bm{I}$. The above settings is derived from~\cite{beck2022goal}. The PDE problem~\eqref{eq:pde_general} is solved with \texttt{DEAL.II}~\cite{dealII92}, using the $\mathcal{Q}_1$ finite element on a 16$\times$16 mesh. 

Recall that the coefficient $a$ involved in the PDE model~\eqref{eq:bvp_general} is given by,
\begin{align*}
	a(\mathbf{x};\omega) = \exp\left(\sum_{j=1}^s y_j \psi_j\right). 
\end{align*}
where the basis $\{\psi_j\}$ are the eigenfunctions to the following Mat\'ern covariance kernel: 
\begin{align}
	C(h) = \frac{1}{2^{\nu-1} \Gamma({\nu})} \left( \sqrt{2\nu} \frac{h}{r} \right)^{\nu} K_{\nu} \left( \sqrt{2\nu} \frac{h}{r} \right),
\end{align}
where we chose $\nu = 4.5$ and $r = 1$. In addition, $\Gamma$ is the Gamma function, and $K_{\nu}$ is the modified Bessel function of the second kind. Let $\lambda_{j, p}$ and $\theta_{j, p}$ be the Fourier coefficients and trigonometric functions in the Fourier expansion of the covariance kernel $C$ on an extended domain $\mathcal{D}_p = [-\gamma, \gamma]^d$, respectively, with $d = 2$. The extension is needed to ensure the positivity of the Fourier coefficients $\lambda_{j, p}$~\cite{bachmayr2018representations}. We let
\begin{align}
	\begin{split}
		\psi_{j, p} &= \lambda_{j,p} \theta_{j, p} \\
		\psi_j &= \left.\psi_{j, p} \right|_{\mathcal{D}}.
	\end{split}
\end{align}

Apart from the trigonometric-basis expansions, we also apply Meyer wavelet functions to construct the basis $\psi_j$ (for a detailed analysis and instructions, see Section~4 in~\cite{bachmayr2018representations}). 

To validate the proposed convergence model~\eqref{eq:Astar_pde}, we also considered the following coefficient:
\begin{align}
	a(\mathbf{x};\omega) = \exp\left(\sum_{j=1}^s y_j \sigma_j \psi_j\right),
\end{align}
where we generated six samples of $\bm{\sigma}$ within each dimension setting. In the first sample, $\bm{\sigma} = \mathbf{1}$. In the second and third samples, $\sigma_j$ are i.i.d. sampled from the uniform distribution $U[1,2]$. In the last three samples, the $\bm{\sigma}$ values of the first three cases are multiplied component-wise by 2. Figure~\ref{fig:example_2_scatter} plots the deteriorated rate $1 - \gamma$ against $\sqrt{\sum_{j=1}^s \sigma_j^2 b_j^2}$. The convergence rate is estimated based on a linear regression between the data from $n=2^6$ to $n=2^{15}$ on a log-log scale, with one RQMC estimator of $R=30$ randomizations. In all considered cases, the rate $1 - \gamma$ exhibits an almost linear dependence on $\sqrt{\sum_{j=1}^s \sigma_j^2 b_j^2}$. Unlike Example 1, the effect of dimensions on the convergence rates is not obvious in both cases, which was unexpected because the QMC points quality deteriorates in higher dimensions. One possible reason for this observation lies in Equation~\eqref{eq:partial_Q_inequality}, where we applied the upper bounds for $Q$ and its derivatives, which conceals the effective dimension. 

\begin{figure}[htbp]
	\centering
	\includegraphics[width = 0.48\textwidth]{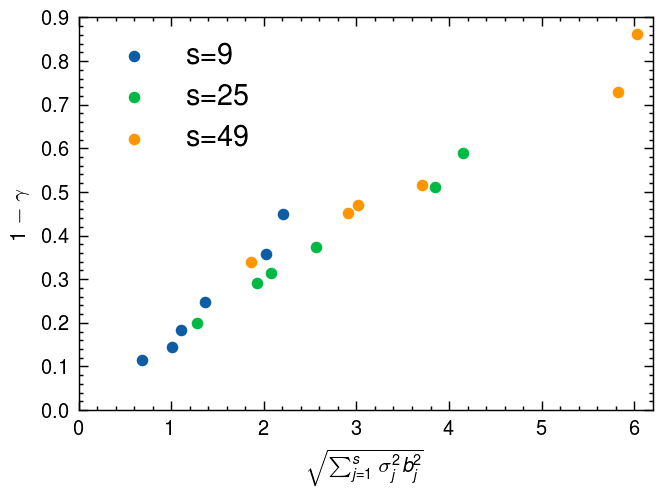}
	\hfill
	\includegraphics[width = 0.48\textwidth]{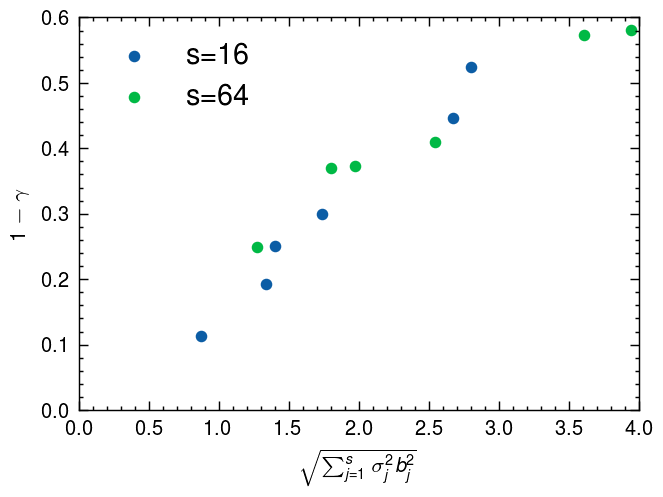}
	\caption{Example 2, the value $(1-\gamma)$ against $\sqrt{\sum_{j=1}^s \sigma_j^2 b_j^2}$ for using the Fourier basis (left) and wavelet-type basis (right) across various dimension and variance settings. The colors represent the dimension $s$, whereas each scatter point with the same color represents a unique variance setting. The dimension-independent effect is observed in the convergence rate in both cases. 
}
\label{fig:example_2_scatter}
\end{figure}

Figure~\ref{fig:fourierwaveletlinfty} compares the $L^\infty$ norm of the basis $ {\psi_{j, p}}$ and $ {\psi_{j}}$ between trigonometric and wavelet-type bases. 
The trigonometric basis is optimal in approximating the random field in the $L^2$ sense. However, the basis $L^{\infty}$ norm, which appears in the convergence model~\eqref{eq:Astar_pde} and \eqref{eq:QMC_convergence_model} is the primary interest. Ideally, $\norm {\psi_{j}}_{L^{\infty} (\mathcal{D}) }$ is desired to converge fast as $j$ increases to reduce the effective dimension. In the nonasymptotic QMC convergence model for elliptic PDEs~\eqref{eq:Astar_pde}, the value $\sum_{j=1}^s b_j^2$ is expected to be small for a fixed dimension $s$. We observed that the values $b_j$ of the wavelet-type basis in the plotted range are smaller than those of the trigonometric basis. The localization properties of the wavelet-type basis make it even more favorable when the function $\psi_{j, p}$ is restricted on $\mathcal{D}$ since $\norm {\psi_{j}}_{L^{\infty} (\mathcal{D}) } \leq \norm {\psi_{j, p}}_{L^{\infty} (\mathcal{D}_p) }$.

\begin{figure}
	\centering
	\includegraphics[width=0.48\linewidth]{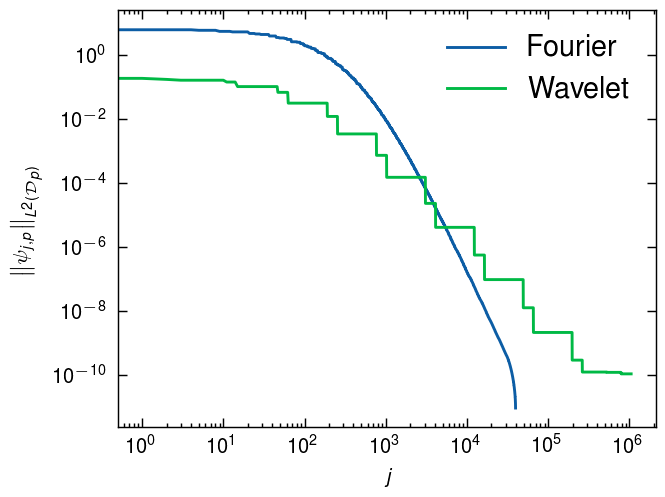}
	\hfill
	\includegraphics[width=0.48\linewidth]{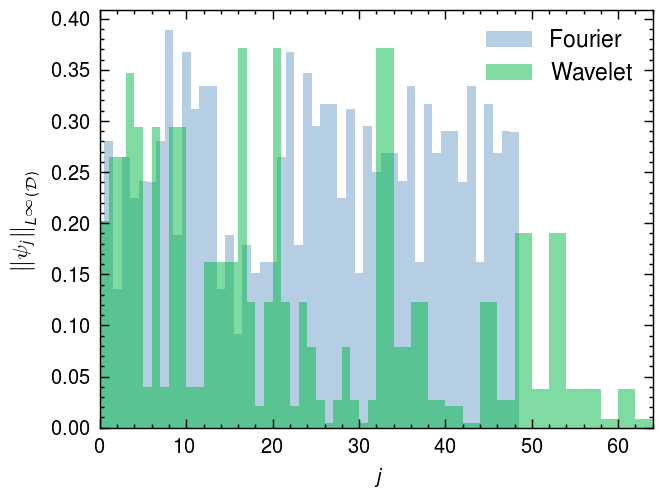}
	\caption{Example 2, a comparison of the $L^2$ norms $\norm{\psi_{j, p}}_{L^2(\mathcal{D}_p)}$ and $L^{\infty}$ norms $\norm{\psi_j}_{L^{\infty}(\mathcal{D}) }$ between the Fourier basis and wavelet-type basis. }
	\label{fig:fourierwaveletlinfty}
\end{figure}

Figure~\ref{fig:example_2_is_wavelet} plots the convergence of the RMSE for the original integrand $g$ and the integrand with IS $g_{\textrm{IS}}$, where the wavelet-type basis is used for two variance settings and two dimension settings. In all cases, both types of IS improve the integrand regularity, reducing the RMSE and improving the convergence rate. Furthermore, the effectiveness of the IS is more significant when the dimension is smaller, as IS has more limitations in higher dimensions to not increase the variance. {Moreover, it has been observed that the IS is more effective when the original integrand has a higher variance, particularly when $\sigma^2$ is larger.}

\begin{figure}[htbp]
	\centering
	\includegraphics[width = 0.48\textwidth]{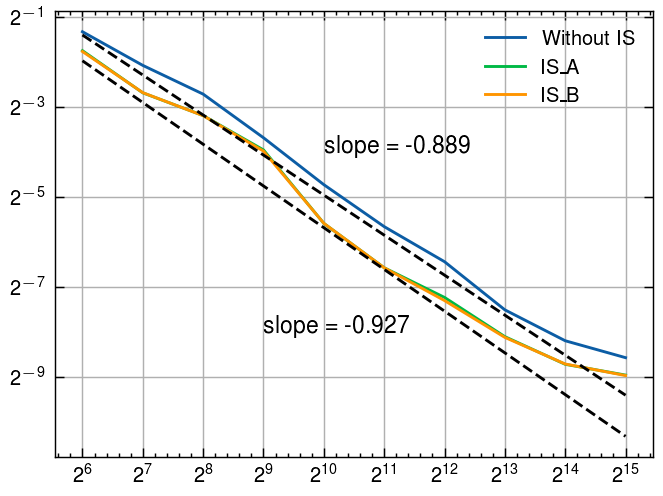}
	\hfill
	\includegraphics[width = 0.48\textwidth]{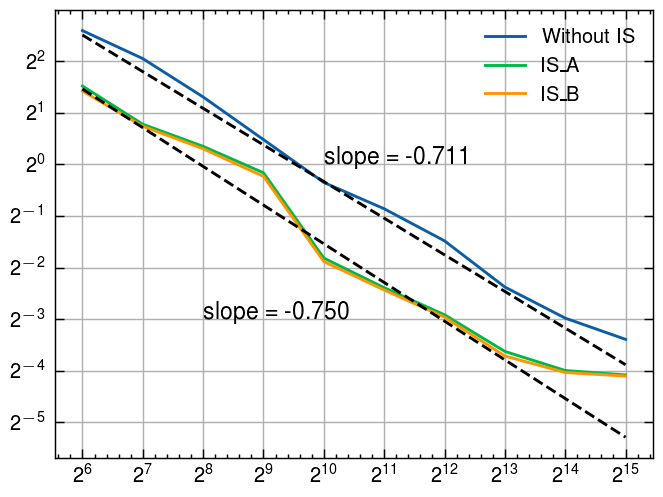}
	\\
	\includegraphics[width = 0.48\textwidth]{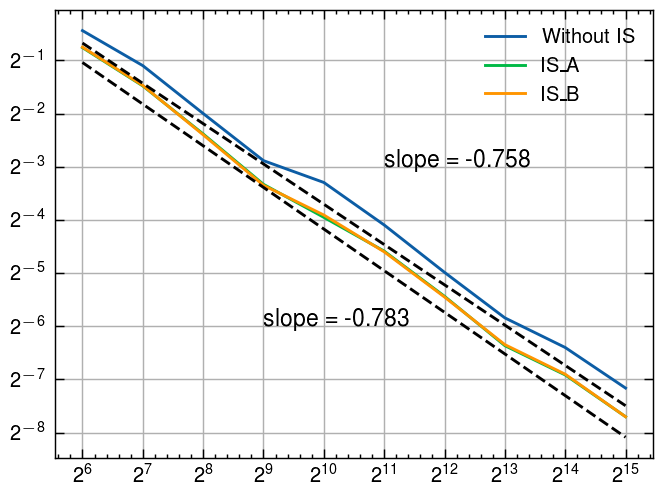}
	\hfill
	\includegraphics[width = 0.48\textwidth]{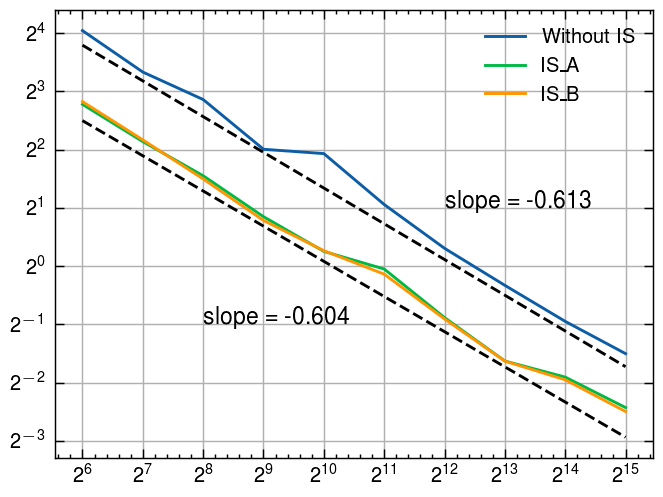}
	
	\caption{Example 2, the root mean squared error with and without two types of IS for $s=16$ (top), $s=64$ (bottom) with $\bm{\sigma} = \bm{1}$ (left) and $\bm{\sigma} = \bm{2}$ (right). The fitted convergence rate is marked in the figure. The random coefficient $a$ is expressed using the wavelet-type basis. }
	\label{fig:example_2_is_wavelet}
\end{figure}

\section{Conclusion}
\label{sec:concl}

In this work, we study the nonasymptotic QMC convergence rate model for functions with finite-dimensional inputs. Specifically, we focus on the expectation of a lognormal random variable and an elliptic PDE problem characterized by lognormal coefficients. Drawing upon the hyperbolic set $K_{n,s}$ introduced in the work of Owen~\cite{Owen06}, we subdivided the integration domain. This division allowed us to split the QMC integration error into two distinct contributions, namely: $2\int_{[0, 1]^s - K_{n, s}} \lvert g - \tilde{g} \rvert$ and $\mathbb{E} \left[ D_n^*(t_1, \dotsc, t_n) \right]V_{HK}(\tilde{g})$. This method replaces the right-hand side of the Koksma–Hlawka inequality $V_{\textrm{HK}}(g) D^*_n(\mathcal{P})$, which is infinite for the unbounded functions considered in this work, with two finite terms. 

Our primary contribution through this work has been deriving an upper bound for the integrand derivative via an optimization problem, denoted as \eqref{eq:optimization_aj}. With this in place, we applied it within the convergence model, culminating in deriving the nonasymptotic QMC convergence rate model. In our quest to verify the proposed model, we presented numerical examples for the two problems above. Moreover, the analytical procedures delineated here can find relevance in varied integration challenges, one notable instance being the option pricing problem within financial scenarios. 

In addition to the above, our work also involved applying two IS distributions, leading us to evaluate their potential to enhance integrand regularity. Notably, despite the distinct support of these distributions, their influence on the integrand appeared broadly analogous.

However, there remain areas that we have not explored in depth. For instance, we have yet to derive the constant $C$ within the hyperbolic set. Delving into its dependence on dimensionality and integrand characteristics could be a promising trajectory for subsequent studies. Moreover, future inquiries could also encompass the random field's truncation error in the random PDE model, where wavelets demonstrate a superior edge over the trigonometric basis, as indicated in~\cite{bachmayr2017sparse}. Furthermore, examining a multilevel setting with adaptivity, as hinted at in~\cite{beck2022goal}, could be beneficial.

Lastly, it is worth noting our choice to employ the Sobol' sequence instead of exploring weighted spaces or designing lattice rules (e.g., \cite{Kuo2012QuasiMonteCF}, \cite{graham2015quasi}). The latter could reduce the worst-case integration error for specific function spaces, potentially leading to a dimension-independent outcome given certain conditions. This combination of lattice rules and nonasymptotic analysis remains a compelling avenue for future exploration.

\section*{Acknowledgement}
\label{sec:ackno}

This publication is based on work supported by the Alexander von Humboldt Foundation and the King Abdullah
University of Science and Technology (KAUST) office of sponsored research (OSR) under Award
No. OSR-2019-CRG8-4033. This work utilized the resources of the Supercomputing Laboratory at King Abdullah University of Science and Technology (KAUST) in Thuwal, Saudi Arabia. The authors thank Christian Bayer, Fabio Nobile and Erik von Schwerin for fruitful discussions. We are also grateful to Arved Bartuska and Michael Samet for proofreading this manuscript. Their suggestions significantly improve the paper's readibility.
We acknowledge the use of the following open-source software
packages: \texttt{deal.II}~\cite{dealII92}.


\appendix

\section[Proof of Thm]{Proof of \ref{lemma:bound for d-dim uniform RV}}
\label{sec:proof_of_lemma_bound_for_d-dim_uniform_RV}

\begin{proof}
	Let $\mathbf{t}_n \sim U[0, 1]^s$. We have
	\begin{equation}
	\begin{split}
	\textnormal{Pr}\left( \prod_{j=1}^s t_n^j \leq C \cdot n^{-r} \right) &= \textnormal{Pr}\left( -2\log \left(\prod_{j=1}^s t_n^j \right) \geq 2r \log(n) - 2\log (C) \right)  \\
	&= \textnormal{Pr}\left( \chi^2_{(2s)} \geq 2r \log(n) - 2\log (C) \right)  \\
	&= \int_{2r \log(n) - 2\log (C)}^{\infty} \frac{z^{s-1} e^{-z/2}}{2^s \Gamma(s)}dz  \\
	&= \int_{r \log(n) - \log (C)}^{\infty} \frac{y^{s-1} e^{-y}}{\Gamma(s)}dy ,
	\label{eq:proof_bound_d-dim_uniform_RV_1}
	\end{split}
	\end{equation}
where $\chi^2_{(2s)}$ denotes a chi-squared distribution with $2s$ degrees of freedom and $\Gamma(s)$ denotes the gamma function evaluated at $s$. Now, we use the upper bound derived in~\cite{pinelis2020exact} for the incomplete gamma function:
	\begin{align}
		\Gamma(s, x) = \int_{x}^{\infty} y^{s-1} e^{-y}dy \leq G_{s}(x)
	\end{align}
	where 
	\begin{align}
		G_s(x) = \frac{(x + b_s)^s - x^s}{s b_s}e^{-x}
	\end{align}
with $b_s = \Gamma(s+1)^{\frac{1}{s-1}}$. Then, we can bound~\eqref{eq:proof_bound_d-dim_uniform_RV_1} by
	\begin{align}
			\int_{r \log(n) - \log (C)}^{\infty} \frac{y^{s-1} e^{-y}}{\Gamma(s)}dy &\leq \frac{G_s(\log(\frac{n^r}{C}))}{\Gamma(s)}\nonumber \\
			&= \frac{\left(\log\left(\frac{n^r}{C}\right) + \Gamma(s+1)^{\frac{1}{s-1}}\right)^s - \left(\log(\frac{n^r}{C})\right)^s}{\Gamma(s) s \Gamma(s+1)^{\frac{1}{s-1}}}\cdot \frac{C}{n^r}.
			\label{eq:proof_bound_d-dim_uniform_RV_2}
	\end{align}
	The last expression in~\eqref{eq:proof_bound_d-dim_uniform_RV_2} is summable, for $n = 1, \dotsc, +\infty$, yielding $\sum_{n=1}^{\infty} \textrm{Pr}(E_n) < \infty$. Thus, by the Borel–Cantelli lemma,
	\begin{align*}
		\textnormal{Pr}\left( E_n \enspace i.o. \right) &= \textnormal{Pr}\left( \bigcap_{n=1}^{\infty} \bigcup_{k = n}^{\infty} E_k \right) \nonumber \\
		&= \lim_{n \to \infty} \textnormal{Pr}\left( \bigcup_{k = n}^{\infty} E_k \right)\\
		&\leq \lim_{n \to \infty} \sum_{k=n}^{\infty} \textnormal{Pr}\left( E_k \right)\\
		&= 0.
	\end{align*}
\end{proof}

\section[Proof of Thm]{Proof of Lemma~\ref{lemma:integration_difference_integrand_sobol_low_variation_extension}}
\label{sec:proof_of_lemma_integration_difference_integrand_sobol_low_variation_extension}
\begin{proof}
We use proof by induction. Let $A_1, \dotsc, A_s < 1$ be distinct. Denote $\mathbf{t} = (t_1, \dotsc, t_s)$. Define the set $K_{\epsilon} \coloneqq \{\mathbf{t} \in [0, 1]^s: \prod_{j=1}^s t_j \geq \epsilon \}$, the polynomial $\phi(a) \coloneqq \prod_{j=1}^s (a - A_j)$. First we propose the hypothesis that
\begin{equation}
	\int_{K_{\epsilon}} \prod_{j=1}^{s} t_j^{-A_j} d\mathbf{t} = \frac{1}{\phi(1)} + \sum_{j=1}^s \frac{\epsilon^{1- A_j}}{(A_j - 1) \phi^{\prime} (A_j)}.
\end{equation}
The base case $s = 1$ is easy to verify. Now we assume the induction hypothesis holds for $s - 1$ dimension, and we will prove it for $s$ dimensions. First we rewrite the integral use change of variable as follows:
\begin{equation}
	\label{eq:b2}
	\begin{split}
		\int_{K_{\epsilon}} \prod_{j=1}^{s} t_j^{-A_j} d\mathbf{t} &= \int_{\sum_{j=1}^s x_j \geq \log \epsilon} \prod_{j=1}^{s} e^{(1-A_j) x_j} dx_1 \dotsm dx_s\\
		&= \int_{\log \epsilon}^{0} e^{x_1(1-A_1)} dx_1 \cdots \int_{\log \epsilon - \sum_{j=1}^{s-2} x_j}^{0} e^{x_{s-1}(1-A_{s-1})} dx_{s-1} \\
		& \quad \cdot \int_{\log \epsilon - \sum_{j=1}^{s-1}}^{0} e^{x_s(1-A_s)} dx_s. 
	\end{split}
\end{equation}
We have
\begin{equation}
	\label{eq:b3}
	\int_{\log \epsilon - \sum_{j=1}^{s-1}}^{0} e^{x_s(1-A_s)} dx_s = \frac{1}{1-A_s} - \frac{\epsilon^{1-A_s}}{1-A_s} e^{(A_s-1) \sum_{j=1}^{s-1} x_j}.
\end{equation}
Substitude~\eqref{eq:b3} into~\eqref{eq:b2} we have
\begin{equation}
	\int_{K_{\epsilon}} \prod_{j=1}^s t_j^{-A_j} = \int_{\sum_{j=1}^{s-1} x_j \geq \log \epsilon} \frac{1}{1 - A_s} \prod_{j=1}^{s-1} e^{x_j(1-A_j)} - \frac{\epsilon^{1-A_s}}{1-A_s} \prod_{j=1}^{s-1} e^{(A_s - A_j)x_j} dx_1 \dotsm dx_{s-1}.
\end{equation}
With the change of variable and induction hypothesis, we have
\begin{equation}
	\label{eq:b5}
	\begin{split}
	&\int_{\sum_{j=1}^{s-1} x_j \geq \log \epsilon} \frac{1}{1 - A_s} \prod_{j=1}^{s-1} e^{x_j(1-A_j)} dx_1 \dotsm dx_{s-1}\\
	 &= \frac{1}{1 - A_s} \int_{\sum_{j=1}^{s-1} x_j \geq \log \epsilon} \prod_{j=1}^{s-1} e^{x_j(1-A_j)} dx_1 \dotsm dx_{s-1}\\
	&= \frac{1}{1 - A_s} \left( \frac{1}{\prod_{j=1}^{s-1} (1 - A_j) } + \sum_{j=1}^{s-1} \frac{\epsilon^{1 - A_j}}{(A_j - 1) \prod_{k = 1, k\neq j}^{k = s-1} (A_j - A_k) } \right)
	\end{split}
\end{equation}
and
\begin{equation}
	\begin{split}
	\int_{\sum_{j=1}^{s-1} x_j \geq \log \epsilon} \prod_{j=1}^{s-1} e^{(A_s - A_j)x_j} dx_1 \dotsm dx_{s-1} &= \int_{\sum_{j=1}^{s-1} x_j \geq \log \epsilon} \prod_{j=1}^{s-1} e^{(1 - (1 - A_s + A_j))x_j} dx_1 \dotsm dx_{s-1}\\
	&= \frac{1}{\prod_{j=1}^{s-1} (A_s - A_j)} + \sum_{j=1}^{s-1} \frac{\epsilon^{A_s - A_j}}{(A_j - A_s) \prod_{k = 1, k\neq j}^{k = s-1} (A_j - A_k)  }.
	\end{split}
\end{equation}
Thus
\begin{equation}
	\label{eq:b7}
	\begin{split}
	&\int_{\sum_{j=1}^{s-1} x_j \geq \log \epsilon} - \frac{\epsilon^{1 - A_s}}{1 - A_s} \prod_{j=1}^{s-1} e^{(A_s - A_j)x_j} dx_1 \dotsm dx_{s-1} \\  &\hspace*{8em} = - \frac{\epsilon^{1-A_s}}{(1-A_s)\prod_{j=1}^{s-1} (A_s - A_j)} - \frac{1}{1-A_s} \sum_{j=1}^{s-1} \frac{\epsilon^{1 - A_j}}{(A_j - A_s) \prod_{k = 1, k\neq j}^{k = s-1} (A_j - A_k)  }.
	\end{split}
\end{equation}
By combining~\eqref{eq:b5} and~\eqref{eq:b7}, we have
\begin{equation}
	\begin{split}
		\int_{K_{\epsilon}} \prod_{j=1}^s t_j^{-A_j} &= \frac{1}{1 - A_s} \left( \frac{1}{\prod_{j=1}^{s-1} (1 - A_j) } + \sum_{j=1}^{s-1} \frac{\epsilon^{1 - A_j}}{(A_j - 1) \prod_{k = 1, k\neq j}^{k = s-1} (A_j - A_k) } \right) \\
		&\hspace*{3em}- \frac{\epsilon^{1-A_s}}{(1-A_s)\prod_{j=1}^{s-1} (A_s - A_j)} - \frac{1}{1-A_s} \sum_{j=1}^{s-1} \frac{\epsilon^{1 - A_j}}{(A_j - A_s) \prod_{k = 1, k\neq j}^{k = s-1} (A_j - A_k)  }\\
		&= \frac{1}{\prod_{j=1}^{s} (1 - A_j) } + \frac{1}{1 - A_s} \sum_{j=1}^{s-1} \frac{\epsilon^{1 - A_j}}{\prod_{k = 1, k\neq j}^{k = s-1} (A_j - A_k) } \left( \frac{1}{A_j - 1} - \frac{1}{A_j - A_s} \right)\\
		&\hspace*{3em}+ \frac{\epsilon^{1-A_s}}{(A_s - 1)\prod_{j=1}^{s-1} (A_s - A_j)}\\
		&= \frac{1}{\phi(1)} + \frac{1}{1 - A_s} \sum_{j=1}^{s-1} \frac{\epsilon^{1 - A_j}}{\prod_{k = 1, k\neq j}^{k = s-1} (A_j - A_k) }  \frac{1 - A_s}{(A_j - 1)(A_j - A_s)}\\
		&\hspace*{3em}+ \frac{\epsilon^{1-A_s}}{(A_s - 1) \phi^{\prime}(A_s)}\\
		&= \frac{1}{\phi(1)} + \sum_{j=1}^s \frac{\epsilon^{1-A_j}}{(A_j - 1) \phi^{\prime}(A_j)}. 
	\end{split}
\end{equation}
Notice that
\begin{equation}
	\int_{[0, 1]^s} \prod_{j=1}^s t_j^{-A_j} d\mathbf{t} = \frac{1}{\phi(1)}.
\end{equation}
Thus we have
\begin{equation}
	\int_{[0, 1]^s - K_{\epsilon}} \prod_{j=1}^s t_j^{-A_j} d\mathbf{t} = \sum_{j=1}^s \frac{\epsilon^{1-A_j}}{(A_j - 1) \phi^{\prime}(A_j)}. 
\end{equation}
When $A_1 = A_2 = \cdots = A_s = A < 1$, we substitude~\eqref{eq:b3} into~\eqref{eq:b2} and get
\begin{equation}
	\int_{K_{\epsilon}} \prod_{j=1}^s t_j^{-A} = \int_{\sum_{j=1}^{s-1} x_j \geq \log \epsilon} \frac{1}{1 - A} \prod_{j=1}^{s-1} e^{x_j(1-A)} - \frac{\epsilon^{1-A}}{1-A} dx_1 \dotsm dx_{s-1}.
\end{equation}
We denote $\mathcal{A}_s \coloneqq \int_{K_{\epsilon}} \prod_{j=1}^s t_j^{-A} dt_1 \cdots dt_s$ and $\lambda_s \coloneqq \int_{\sum_{j=1}^{s} x_j \geq \log \epsilon} dx_1 \dotsm dx_{s}$, we have the following:
\begin{equation}
	\begin{split}
	\mathcal{A}_s &= \frac{1}{1-A} \mathcal{A}_{s-1} - \frac{\epsilon^{1-A}}{1-A} \lambda_{s-1}\\
	&= \frac{1}{(1-A)^2} \mathcal{A}_{s-2} - \frac{\epsilon^{1-A}}{1-A} \lambda_{s-1} - \frac{\epsilon^{1-A}}{(1-A)^2} \lambda_{s-2}\\
	& \hspace*{0.5em}  \vdots \\
	&= \frac{1}{(1-A)^{s-1}} \mathcal{A}_{1} - \sum_{k=1}^{s-1} \frac{\epsilon^{1-A}}{(1-A)^{k}} \lambda_{s-k}\\
	&= \frac{1 - \epsilon^{1-A}}{(1-A)^{s}} - \sum_{k=1}^{s-1} \frac{\epsilon^{1-A}}{(1-A)^{k}} \cdot \frac{\log^{s-k} \epsilon}{(s-k)!},
	\end{split}
\end{equation}
where we derive the last line by noticing $A_1 = \frac{1 - \epsilon^{1-A}}{1-A}$ and the measure of the simplex $\lambda_s = \frac{\log^{s} \epsilon}{s!}$. Thus,
\begin{equation}
	\int_{[0, 1]^s - K_{\epsilon}} \prod_{j=1}^s t_j^{-A} d\mathbf{t} = - \sum_{k=1}^{s} \frac{\epsilon^{1-A}}{(1-A)^{k}} \cdot \frac{\log^{s-k} \epsilon}{(s-k)!}. 
\end{equation}
 This concludes the proof. 
\end{proof}

\section[Proof of Thm]{Proof of \ref{lemma:first_order_taylor_approximation_local_rate_A}}
\label{sec:proof_of_lemma_first_order_taylor_approximation_local_rate_A}
\begin{proof}
	We take the logarithm of $h(\mathbf{v})$:
	\begin{align}
	\log {h(\mathbf{v})} = \sum_{j=1}^s \sigma_j \sqrt{-2 \log ({v_j})}.
	\label{eq:inequality_local_rate_A_log}
	\end{align}
	We further apply the change of variable to simplify the notation. For $\mathbf{v}_c \in \partial K_{n,s}$, let $\mathbf{z} = -\log \mathbf{v}$, and $\mathbf{z}_c = -\log \mathbf{v}_c$. By a {first-order Taylor} approximation of~\eqref{eq:inequality_local_rate_A_log},
	\begin{align}
	\begin{split}
	\sum_{j=1}^s \sigma_j \sqrt{-2 \log ({v_j})} &=
	\sum_{j=1}^s \sigma_j \sqrt{2 z_j}\\
	 &\leq \nabla_{\mathbf{z} = \mathbf{z}_c} \left(\sum_{j=1}^s \sigma_j \sqrt{2 {z_j}} \right) \cdot (\mathbf{z} - \mathbf{z}_c) + \left(\sum_{j=1}^s \sigma_j \sqrt{2 z_c^j} \right) \\
	&= \sum_{j=1}^s \frac{\sigma_j}{\sqrt{2 z_c^j}} {z}_j + \sum_{j=1}^s \sigma_j \sqrt{\frac{z_c^j}{2}}
	\end{split}
	\label{eq:inequality_local_approximation_A}
	\end{align}
	where we use the convexity. By taking the exponential of~\eqref{eq:inequality_local_approximation_A} and comparing with the Equation~\eqref{eq:inequality_local_rate_A}, we deduce that
	\begin{align}
	{A}_j &= \frac{\sigma_j}{\sqrt{2z_c^j}} = \frac{\sigma_j}{\sqrt{-2 \log v_c^j}} = \frac{\sigma_j}{\sqrt{-2 \log (1-t_c^j)}}\\
	B(\mathbf{v}_c)&= \prod_{j=1}^{s} \exp\left( \sigma_j \sqrt{\frac{-\log (1 - t_c^j)}{2}} \right). 
	\end{align}
\end{proof}

\section[Lagrangian]{Lagrangian to Optimization problem \eqref{eq:optimization_aj_ex_1}}
\label{sec:lagrangian_to_optimization_problem_3.7}
The Lagrangian is given by
\begin{align}
	\begin{split}
		\mathcal{L} &= \log B(\mathbf{v}_c) +  \sum_{j=1}^s -A_j \log v_c^j + \lambda \left(\sum_{j=1}^s -\log v_c^j - \log \delta^{-1} \right)\\
		&= \sum_{j=1}^s 2 \sigma_j \sqrt{\frac{-\log v_c^j }{2}} + \lambda \left(\sum_{j=1}^s -\log v_c^j - \log\delta^{-1} \right).
	\end{split}
\end{align}
The optimizer $-\log {v_c^j}^*$ for Equation~\eqref{eq:optimization_aj_ex_1} is given by
\begin{align}
	-\log {v_c^j}^* = \frac{\sigma_j^2}{\sum_{k=1}^s \sigma_k^2} \log \delta^{-1}. 
\end{align}

\section[Derivation of $A^*$]{Derivation of $A^*$ in Section~\ref{subsec:qmc4pde}}
\label{sec:derivation_of_astar_pde}
From~\cite{graham2015quasi}, an upper bound for the derivative of $Q$ w.r.t. the random variable $\mathbf{y}$ is given by
\begin{align}
\left\lvert {\partial^{\mathfrak{u}}}Q({\mathbf{y}}) \right\rvert &\leq \lVert \mathcal{G} \rVert_{V^\prime} \lVert \partial^{\mathfrak{u}} u(\cdot, \mathbf{y}) \rVert_V\nonumber \\
&\leq \frac{\lvert \mathfrak{u} \rvert !}{(\log 2)^{\lvert \mathfrak{u} \rvert}} \left( \prod_{j \in \mathfrak{u}} b_j \right) \underbrace{\lVert f \rVert_{V^\prime} \lVert \mathcal{G} \rVert_{V^\prime}}_{K^*} \prod_{j=1}^s \exp(b_j \lvert {\mathrm{y}}_j \rvert).
\label{eq:partial_Q_inequality}
\end{align}
Consider $\mathbf{t} \in [0, \frac{1}{2}]^s$, the derivative of the integrand $g$ is given as follows:
\begin{align}
\label{eq:derivative_g_ex2}
\begin{split}
\abs*{ {\partial }^{{\mathfrak{u}}}g(\mathbf{t})} &= \abs*{\partial^{\mathfrak{u}} Q(\Phi^{-1}(\mathbf{t}))} = \abs{\partial^{\mathfrak{u}}Q} \cdot \prod_{j \in \mathfrak{u}} \abs*{\partial^j \Phi^{-1} (t_j)}\\
&\leq K^* \frac{\lvert \mathfrak{u} \rvert !}{(\log 2)^{\lvert \mathfrak{u} \rvert}} \left(\prod_{j \in \mathfrak{u}} b_j \right) \cdot \prod_{j=1}^s \exp(b_j \lvert \mathrm{y}_j \rvert) \prod_{j\in \mathfrak{u}} \abs*{\partial^j \Phi^{-1} ({t}_j)}\\
&\leq K^* \frac{\lvert \mathfrak{u} \rvert !}{(\log 2)^{\lvert \mathfrak{u} \rvert}} \left(\prod_{j \in \mathfrak{u}} b_j \right) \prod_{j=1}^s \exp\left( \abs{\bar{\epsilon}} \sum_{j=1}^s b_j \right) \exp\left(b_j {\sqrt{-2\log(t_j)}}\right)\\
&\hspace*{15em} \cdot \prod_{j\in \mathfrak{u}} \sqrt{2\pi} \exp\left(\frac{\bar{\epsilon}^2 }{2}\right) t_j^{-(1+\abs{\bar{\epsilon}})}, 
\end{split}
\end{align}
where we apply the upper bound of $\abs{\partial \Phi^{-1}}$ in~\eqref{eq:derivative_phi_-1_upper_bound}. The cases when $\mathbf{t} \notin [0, \frac{1}{2}]^s$ can be derived by symmetry. The upper bound~\eqref{eq:derivative_g_ex2} follows a similar behavior as~\eqref{eq:upper_bound_partial_derivative_g_ex1} of Example 1, and we skip the rest of derivations for the RQMC error bound. 

\section[Study of ${\partial^{\mathfrak{u}}}g_{\mathrm{IS}}$]{Detailed derivations of ${\partial^{\mathfrak{u}}}g_{\mathrm{IS}}$ of the first example in Section~\ref{subsec:first_integral}}
\label{sec:derivation_of_partial_g_is}
Let us define $\varphi(\mathbf{t}) = \prod_{j=1}^s {\alpha_j} \exp\left( -\frac{(\Phi^{-1}({t_j}))^2}{2} ({\alpha}_j^2 - 1) \right)$ for the simplicity of notations. We will analyze the mixed first-order derivative ${\partial^{\mathfrak{u}}}g_{\mathrm{IS}}$.  Specifically, by the Leibniz product rule, we have
\begin{align}
\begin{split}
\abs*{ {\partial^{\mathfrak{u}}}g_{\mathrm{IS}} (\mathbf{t})} &=\abs*{ \sum_{\mathfrak{z} \subseteq \mathfrak{u}} \frac{\partial}{\partial t_{\mathfrak{u}- \mathfrak{z}}} \nu (\bar{\mathbf{y}}) \cdot {\partial^{\mathfrak{z}}} \varphi(\mathbf{t}) }.
\end{split}
\label{eq:g_pde_is_1}
\end{align}
We have
\begin{align}
\frac{\partial}{\partial t_{\mathfrak{u}- \mathfrak{z}}} \nu (\bar{\mathbf{y}}) = \frac{\partial}{\partial \bar{\mathrm{y}}_{\mathfrak{u}- \mathfrak{z}}} \nu (\bar{\mathbf{y}}) \cdot \prod_{j \in \mathfrak{u}- \mathfrak{z}} \alpha_j \frac{\partial \mathrm{y}_j}{\partial t_j}, 
\label{eq:g_pde_is_2}
\end{align}
and
\begin{align}
\begin{split}
{\partial^{\mathfrak{z}}} \varphi(\mathbf{t}) &= \prod_{j=1}^s {\alpha_j} \exp\left( -\frac{(\Phi^{-1}({t_j}))^2}{2} ({\alpha}_j^2 - 1) \right) \cdot \prod_{k \in \mathfrak{z}} -\frac{\alpha_k^2 - 1}{2} \cdot \frac{\partial \left(\Phi^{-1}(t_k)\right)^2}{\partial t_k}. 
\end{split}
\label{eq:g_pde_is_3}
\end{align}
In Example 1 (i.e., when $\nu = \exp \circ \boldsymbol{\sigma}^T $), we have
\begin{align}
	\partial^{\mathfrak{u}- \mathfrak{z}} \nu(\bar{\mathbf{y}}) = \exp\left( \sum_{j=1}^s \sigma_j \bar{\mathrm{y}}_j \right) \cdot \prod_{k \in \mathfrak{u}- \mathfrak{z}} \sigma_k.
	\label{eq:g_pde_is_3_1}
\end{align}
Notice that we have the following asymptotic approximations as $t \to 0$:
\begin{align}
	\label{eq:approximations_phi_-1_asymptotic}
	\begin{split}
		\Phi^{-1}(t) &= -\sqrt{-2 \log t} + o(1),\\
		\frac{\partial \Phi^{-1}(t)}{\partial t} &= \frac{\sqrt{2\pi}}{t} \exp\left(o(1) - o(1)\sqrt{-2\log t}\right) = \mathcal{O}(t^{-1})
	\end{split}
\end{align}
By combining equations~\eqref{eq:g_pde_is_1} to \eqref{eq:approximations_phi_-1_asymptotic}, when $\mathbf{t} \in [0, \frac{1}{2}]^s$, we have
\begin{align}
\begin{split}
\abs*{ {\partial^{\mathfrak{u}}} g_{\mathrm{IS}} (\mathbf{t})} &=\abs*{ \sum_{\mathfrak{z} \subseteq \mathfrak{u}}   \frac{\partial}{\partial t_{\mathfrak{u}- \mathfrak{z}}} \nu (\bar{\mathbf{y}}) \cdot {\partial^\mathfrak{u}} \varphi(\mathbf{t}) } \\
&\leq \prod_{j=1}^s {\alpha_j} \exp\left( -\frac{(\Phi^{-1}({t_j}))^2}{2} ({\alpha}_j^2 - 1) \right) \cdot \prod_{j=1}^s \exp(\sigma_j \lvert \alpha_j \Phi^{-1}(t_j) \rvert)\\
&\cdot  \sum_{\mathfrak{z} \preceq \mathfrak{u}}  \prod_{j \in \mathfrak{u}- \mathfrak{z}} \sigma_j \alpha_j \frac{\partial \mathrm{y}_j}{\partial t_j}
\cdot \prod_{k \in \mathfrak{z}} -\frac{\alpha_k^2 - 1}{2} \cdot {\frac{\partial \left(\Phi^{-1}(t_k) \right)^2}{\partial t_k }}\\
&= \prod_{j=1}^s  {\alpha_j}  t_j^{{\alpha}_j^2 - 1} \exp\left(o(1) \sqrt{-2\log t} - o(1)\right) \cdot \exp\left(\alpha_j \sigma_j {\sqrt{-2\log (t_j)} } + o(1) \right)\\
& \cdot \sum_{\mathfrak{z} \preceq \mathfrak{u}}   \prod_{j \in \mathfrak{u} - \mathfrak{z}} \sigma_j  \alpha_j \mathcal{O}(t_j^{-1}) \prod_{k \in \mathfrak{z}}  \mathcal{O}(t_k^{-1}) \left(\sqrt{-2 \log t_k} + o(1)\right)\\
&= \prod_{j=1}^s \mathcal{O} (t_j^{-\mathbbm{1}_{j \in \mathfrak{u}} + \alpha_j^2 - 1 - \delta}) \quad  t_j \to 0, \ j = 1, \dotsc, s, \delta > 0.
\end{split}
\label{eq:isqmc_pde_A_1}
\end{align}
where we have used the asymptotic approximation of $\Phi^{-1}$ to focus on derivative at the boundary. 

\section[Study of ${\partial^{\mathfrak{u}}}g_{\mathrm{IS}}$]{Detailed derivations of ${\partial^{\mathfrak{u}}}g_{\mathrm{IS}}$ of the second example in Section~\ref{subsec:first_integral}}
\label{sec:derivation_of_partial_g_is_2}
In Example 2, $\nu = Q$, we apply the following upper bound of $\abs*{\partial^{{\mathfrak{u} - \mathfrak{z}}}Q(\bar{\mathbf{y}})}$:
\begin{align}
\abs*{\partial^{{\mathfrak{u} - \mathfrak{z}}}Q(\bar{\mathbf{y}})} \leq \frac{\lvert \mathfrak{u} - \mathfrak{z} \rvert !}{(\log 2)^{\lvert \mathfrak{u} - \mathfrak{z} \rvert}} \left( \prod_{j \in \mathfrak{u} - \mathfrak{z}} b_j \right) K^* \prod_{j=1}^s \exp(b_j \lvert \bar{\mathrm{y}}_j \rvert). 
\label{eq:g_pde_is_4}
\end{align}
Combining equations~\eqref{eq:g_pde_is_1} to \eqref{eq:g_pde_is_3} and \eqref{eq:g_pde_is_4}, when $\mathbf{t} \in [0, \frac{1}{2}]^s$, we obtain
\begin{align}
\begin{split}
\abs*{ {\partial^{\mathfrak{u}}} g_{\mathrm{IS}} (\mathbf{t})} &=\abs*{ \sum_{\mathfrak{z} \subseteq \mathfrak{u}}  {\partial^{\mathfrak{u}- \mathfrak{z}}} Q(\mathbf{t}) \cdot {\partial^\mathfrak{u}} \varphi(\mathbf{t}) } \\
&\leq K^* \prod_{j=1}^s {\alpha_j} \exp\left( -\frac{(\Phi^{-1}({t_j}))^2}{2} ({\alpha}_j^2 - 1) \right)  \cdot \prod_{j=1}^s \exp(b_j \lvert \alpha_j \Phi^{-1}(t_j) \rvert)\\
&\cdot \sum_{\mathfrak{z} \preceq \mathfrak{u}} \frac{\lvert \mathfrak{u} - \mathfrak{z} \rvert !}{(\log 2)^{\lvert \mathfrak{u} - \mathfrak{z} \rvert}} \left( \prod_{j \in \mathfrak{u} - \mathfrak{z}} b_j \right) \prod_{j \in \mathfrak{u}- \mathfrak{z}} \alpha_j \frac{\partial \mathrm{y}_j}{\partial t_j} 
\cdot \prod_{k \in \mathfrak{z}} -\frac{\alpha_k^2 - 1}{2} \cdot {\frac{\partial \left(\Phi^{-1}(t_k) \right)^2}{\partial t_k}}\\
&= K^* \prod_{j=1}^s {\alpha_j}  t_j^{{\alpha}_j^2 - 1} \exp\left(o(1) \sqrt{-2\log t} - o(1)\right)  \cdot \exp\left( \alpha_j b_j {\sqrt{-2\log (t_j)} } + o(1) \right)\\
& \cdot \sum_{\mathfrak{z} \preceq \mathfrak{u}}  \frac{\lvert \mathfrak{u} - \mathfrak{z} \rvert !}{(\log 2)^{\lvert \mathfrak{u} - \mathfrak{z} \rvert}} \prod_{j \in \mathfrak{u} - \mathfrak{z}} b_j  \alpha_j \mathcal{O} (t_j^{-1}) \prod_{k \in \mathfrak{z}} \mathcal{O}(t_k^{-1}) \left(\sqrt{-2 \log t_k} + o(1)\right)\\
&=  \mathcal{O} \left(\prod_{j=1}^s t_j^{-\mathbbm{1}_{j \in \mathfrak{u}} + \alpha_j^2 - 1 - \delta}\right) \quad  t_j \to 0, \ j = 1, \dotsc, s, \delta > 0.
\end{split}
\label{eq:isqmc_pde_A_2}
\end{align}
We observe that the upper bound~\eqref{eq:isqmc_pde_A_2} shares similar structures with~\eqref{eq:isqmc_pde_A_1}. 


\section[Study of ${\partial^{\mathfrak{u}}}g_{\mathrm{IS}}$]{Detailed derivations of ${\partial^{\mathfrak{u}}}g_{\mathrm{IS}}$ in Section~\ref{subsec:second_integral}}
\label{sec:derivation_of_partial_g_is_beta}
Next, we study the mixed first-order derivative of $g_{\textrm{IS}}$:
\begin{align}
\begin{split}
\abs*{ \frac{\partial}{\partial w_{\mathfrak{u}}}g_{\mathrm{IS}} (\mathbf{w})} &=\abs*{ \sum_{\mathfrak{z} \subseteq \mathfrak{u}}  \frac{\partial}{\partial w_{\mathfrak{u}- \mathfrak{z}}} Q \left(\Phi^{-1} (\Phi_{{\beta}}^{-1} (\mathbf{w}))\right) \cdot \frac{\partial}{\partial w_\mathfrak{z}} \varphi(\mathbf{w}) },\\
\end{split}
\end{align}
where $\varphi(\mathbf{w}) = \frac{1}{\rho_{\beta}(\Phi_{{\beta}}^{-1} (\mathbf{w}))}$. We have
\begin{align}
\frac{\partial}{\partial {w}_{\mathfrak{u}- \mathfrak{z}}} Q(\mathbf{w}) = \frac{\partial}{\partial \bar{\mathrm{y}}_{\mathfrak{u}- \mathfrak{z}}}Q(\bar{\mathbf{y}}) \cdot \prod_{j \in \mathfrak{u}- \mathfrak{z}} \frac{\partial \bar{\mathrm{y}}_j}{\partial t_j} \frac{\partial t_j}{\partial w_j}
\label{eq:b_is_1}
\end{align}
and
\begin{align}
\begin{split}
\frac{\partial}{\partial w_\mathfrak{z}} \varphi(\mathbf{w}) &= \frac{\partial}{\partial w_\mathfrak{z}} \prod_{j=1}^s \frac{1}{\rho_{\beta_j} (\Phi_{{\beta}_j}^{-1} ({w}_j))}\\
&= \prod_{k \in 1:s- \mathfrak{z}} \frac{1}{\beta_k} (2w_k)^{\frac{1}{\beta_k} - 1} \cdot \prod_{j \in \mathfrak{z}} \frac{2(1-\beta_j)}{\beta_j^2} (2w_j)^{\frac{1}{\beta_j}- 2}, 
\end{split}
\label{eq:b_is_2}
\end{align}
where we define $\bar{\mathbf{y}} = \Phi^{-1}(\mathbf{t})$ and $\mathbf{t} = \Phi^{-1}_\beta (\mathbf{w})$. We only consider the case where $0 \leq w_j < 1/2, \forall j = 1, \dotsc, s$ for simplicity, as other cases can be induced by symmetry. We have
\begin{align}
\frac{\partial t_j}{\partial w_j} &=  \frac{1}{\beta_j} (2w_j)^{\frac{1}{\beta_j} - 1}, 
\label{eq:b_is_3}
\end{align}
and 
\begin{align}
\frac{\partial \bar{\mathrm{y}}_j}{\partial t_j} &= \frac{\partial \Phi^{-1}(t_j)}{\partial t_j} = \mathcal{O}(t_j^{-1}) \quad \text{as } t_j \to 0.
\label{eq:b_is_4}
\end{align}
Combining~\eqref{eq:g_pde_is_4} and \eqref{eq:b_is_1}-\eqref{eq:b_is_4}, we obtain
\begin{align}
\begin{split}
\abs*{ \frac{\partial}{\partial w_{\mathfrak{u}}}g_{\mathrm{IS}}(\mathbf{w})} &=\abs*{ \sum_{\mathfrak{z} \subseteq \mathfrak{u}}  {\partial^{{\mathfrak{u}- \mathfrak{z}}}} Q\left(\Phi^{-1} (\Phi_{{\beta}}^{-1} (\mathbf{w})) \right) \cdot {\partial^{\mathfrak{z}}} \varphi(\mathbf{w}) } \\
&= \abs*{ \sum_{\mathfrak{z} \subseteq \mathfrak{u}}  \frac{\partial}{\partial \bar{\mathbf{y}}_{\mathfrak{u}- \mathfrak{z}}} Q(\bar{\mathbf{y}}) \cdot \prod_{j\in \mathfrak{u}- \mathfrak{z}} \frac{\partial \bar{\mathrm{y}}_j }{\partial t_j} \frac{\partial t_j}{\partial w_j} \cdot \frac{\partial}{\partial w_\mathfrak{z}} \varphi(\mathbf{w}) }\\
&\leq \abs*{ \sum_{\mathfrak{z} \subseteq \mathfrak{u}}  \frac{\lvert \mathfrak{u} - \mathfrak{z} \rvert !}{(\log 2)^{\lvert \mathfrak{u} - \mathfrak{z} \rvert}} \left( \prod_{j \in \mathfrak{u} - \mathfrak{z}} b_j \right) K^* \prod_{j=1}^s \exp\left(b_j \left\lvert \Phi^{-1}\left( (2w_j)^{\frac{1}{\beta}} \cdot \frac{1}{2} \right) \right\rvert\right)}\\
& \cdot  \prod_{j \in \mathfrak{u} - \mathfrak{z}} \frac{\partial \bar{\mathrm{y}}_j}{\partial t_j} \frac{\partial t_j}{\partial w_j}   \cdot \prod_{k \in 1:s- \mathfrak{z}} \frac{1}{\beta_k} (2w_k)^{\frac{1}{\beta_k} - 1} \cdot \prod_{j \in \mathfrak{z}} \frac{2(1-\beta_j)}{\beta_j^2} (2w_j)^{\frac{1}{\beta_j}- 2}\\
& = { \sum_{\mathfrak{z} \subseteq \mathfrak{u}}  \frac{\lvert \mathfrak{u} - \mathfrak{z} \rvert !}{(\log 2)^{\lvert \mathfrak{u} - \mathfrak{z} \rvert}} \left( \prod_{j \in \mathfrak{u} - \mathfrak{z}} b_j \right) K^* \prod_{j=1}^s \exp\left(b_j  \sqrt{-2\log \left( t_j \right)} +o(1) \right)}\\
& \cdot  \prod_{j \in \mathfrak{u} - \mathfrak{z}}  \mathcal{O}{(t_j^{-1})} \frac{1}{\beta_j} (2w_j)^{\frac{1}{\beta_j} - 1}  \cdot \prod_{k \in 1:s- \mathfrak{z}} \frac{1}{\beta_k} (2w_k)^{\frac{1}{\beta_k} - 1} \cdot \prod_{j \in \mathfrak{z}} \frac{2(1-\beta_j)}{\beta_j^2} (2w_j)^{\frac{1}{\beta_j}- 2}\\
& = { \sum_{\mathfrak{z} \subseteq \mathfrak{u}} K^*_{\mathfrak{u}, \mathfrak{z}} \prod_{j=1}^s \exp\left(b_j  \sqrt{-2\log \left( t_j \right)} +o(1) \right)}\\
& \cdot  \prod_{j \in \mathfrak{u} - \mathfrak{z}}  \mathcal{O}  (w_j^{- 1} ) \cdot \prod_{k \in 1:s- \mathfrak{z}} \frac{1}{\beta_k} (2w_k)^{\frac{1}{\beta_k} - 1} \cdot \prod_{j \in \mathfrak{z}} \frac{2(1-\beta_j)}{\beta_j^2} (2w_j)^{\frac{1}{\beta_j}- 2}\\
&\leq \sum_{\mathfrak{z} \subseteq \mathfrak{u}} K^*_{\mathfrak{u}, \mathfrak{z}} \prod_{j=1}^s \exp\left(b_j  \sqrt{-\frac{2}{\beta_j} \log \left( w_j \right)} +o(1) \right) \cdot  \prod_{j \in \mathfrak{u} } \mathcal{O} (w_j^{- 1})  \cdot \prod_{k=1}^s \mathcal{O} (w_k^{\frac{1}{\beta_k} - 1})\\
&= \mathcal{O} \left( \prod_{j=1}^s w_j^{-\mathbbm{1}_{j \in \mathfrak{u}} + \frac{1}{\beta_k} - 1 - \delta} \right) \quad w_j \to 0, j = 1,\dotsc, s, \textrm{for } \delta > 0.
\end{split}
\end{align}

\bibliographystyle{siamplain}
\bibliography{references, bibliography_QMC_theory,bibliography_QMC,bibliography_QMC_PDE,bibliography_QMC_finance}

\end{document}